\documentclass[a4paper,reqno]{amsart}

\usepackage{amssymb,upref}
\usepackage[mathcal]{euscript}
\usepackage[cmtip,all]{xy}

\newcommand{\bydef}{:=}
\newcommand{\defby}{=:}

\newcommand{\vphi}{\varphi}
\newcommand{\veps}{\varepsilon}

\newcommand{\wh}[1]{\widehat{#1}}

\newcommand{\id}{\mathrm{id}}

\newcommand{\espan}[1]{\mathrm{span}\left\{#1\right\}}

\newcommand{\tr}{\mathrm{tr}}

\newcommand{\bi}{\mathbf{i}}

\newcommand{\cA}{\mathcal{A}}
\newcommand{\cB}{\mathcal{B}}
\newcommand{\cC}{\mathcal{C}}
\newcommand{\cD}{\mathcal{D}}
\newcommand{\cG}{\mathcal{G}}

\newcommand{\cH}{\mathcal{H}}

\newcommand{\cJ}{\mathcal{J}}
\newcommand{\cK}{\mathcal{K}}
\newcommand{\cL}{\mathcal{L}}

\newcommand{\cO}{\mathcal{O}}

\newcommand{\cS}{\mathcal{S}}
\newcommand{\cT}{\mathcal{T}}

\newcommand{\frg}{{\mathfrak g}}
\newcommand{\frh}{{\mathfrak h}}

\newcommand{\frstr}{\mathfrak{str}}

\newcommand{\ZZ}{\mathbb{Z}}

\newcommand{\FF}{\mathbb{F}}

\DeclareMathOperator{\rad}{\mathrm{rad}}
\DeclareMathOperator{\Hom}{\mathrm{Hom}}
\DeclareMathOperator{\End}{\mathrm{End}}

\DeclareMathOperator{\Aut}{\mathrm{Aut}}
\DeclareMathOperator{\inaut}{\mathrm{Int}}
\DeclareMathOperator{\Der}{\mathrm{Der}}

\DeclareMathOperator{\Fix}{\mathrm{Fix}}
\DeclareMathOperator{\Int}{\mathrm{int}}
\DeclareMathOperator{\Centr}{\mathrm{Cent}}

\newcommand{\Lie}{\mathrm{Lie}}
\newcommand{\ad}{\mathrm{ad}}

\newcommand{\Gl}{\mathfrak{gl}}
\newcommand{\Sl}{\mathfrak{sl}}

\newcommand{\Sp}{\mathfrak{sp}}

\newcommand{\frsl}{{\mathfrak{sl}}}
\newcommand{\frsp}{{\mathfrak{sp}}}
\newcommand{\frso}{{\mathfrak{so}}}

\newcommand{\frgl}{{\mathfrak{gl}}}

\newcommand{\tri}{\mathfrak{tri}}

\newcommand{\GL}{\mathrm{GL}}
\newcommand{\SL}{\mathrm{SL}}
\newcommand{\PSL}{\mathrm{PSL}}
\newcommand{\PGL}{\mathrm{PGL}}

\newcommand{\SP}{\mathrm{Sp}}
\newcommand{\PSP}{\mathrm{PSp}}

\newcommand{\Ind}{\mathrm{Ind}}

\newcommand{\subo}{_{\bar 0}}
\newcommand{\subuno}{_{\bar 1}}
\newcommand{\fw}{\pi}

\newtheorem{theorem}{Theorem}
\newtheorem{proposition}[theorem]{Proposition}
\newtheorem{lemma}[theorem]{Lemma}
\newtheorem{corollary}[theorem]{Corollary}

\theoremstyle{definition}

\newtheorem{remark}[theorem]{Remark}

\numberwithin{theorem}{section}
\numberwithin{equation}{section}

\newenvironment{romanenumerate}
 {\begin{enumerate}
 
 }{\end{enumerate}}

\begin{document}

\title{Gradings on modules over Lie algebras of E types}

\author[C.~Draper]{Cristina Draper${}^*$}
\address{Departamento de Matem\'{a}tica Aplicada, Escuela de las Ingenier\'{\i}as, Universidad de M\'{a}laga,
Ampliaci\'{o}n Campus de Teatinos, 29071 M\'{a}laga, Spain}
\email{cdf@uma.es}
\thanks{${}^*$ Supported by the Spanish Ministerio de Econom\'{\i}a y Competitividad---Fondo Europeo de
Desarrollo Regional (FEDER) MTM2013-41208-P, and by the Junta de Andaluc\'{\i}a grants FQM-336 and FQM-7156, with FEDER funds}

\author[A.~Elduque]{Alberto Elduque${}^\star$}
\address{Departamento de Matem\'{a}ticas
 e Instituto Universitario de Matem\'aticas y Aplicaciones,
 Universidad de Zaragoza, 50009 Zaragoza, Spain}
\email{elduque@unizar.es}
\thanks{${}^\star$ Supported by the Spanish Ministerio de Econom\'{\i}a y Competitividad---Fondo Europeo de Desarrollo Regional (FEDER) MTM2013-45588-C3-2-P, and by the Diputaci\'on General de Arag\'on---Fondo Social Europeo (Grupo de Investigaci\'on de \'Algebra)}

\author[M. Kochetov]{Mikhail Kochetov${}^\dagger$}
\address{Department of Mathematics and Statistics,
 Memorial University of Newfoundland,
 St. John's, NL, A1C5S7, Canada}
\email{mikhail@mun.ca}
\thanks{${}^\dagger$ Supported by Discovery Grant 341792-2013 of the Natural Sciences and Engineering Research Council (NSERC) of Canada}



\date{}

\begin{abstract}
For any grading by an abelian group $G$ on the exceptional simple Lie algebra $\cL$ of type $E_6$ or $E_7$ over an algebraically closed field of characteristic zero, we compute the graded Brauer invariants of simple finite-dimensional modules, thus completing the computation of these invariants for simple finite-dimensional Lie algebras. This yields the classification of finite-dimensional $G$-graded simple $\cL$-modules, as well as necessary and sufficient conditions for a finite-dimensional $\cL$-module to admit a $G$-grading compatible with the given $G$-grading on $\cL$.
\end{abstract}

\maketitle



\section{Introduction}\label{se:intro}

Gradings have always played an important r\^{o}le in the theory of Lie algebras and their representations. It is sufficient to mention the root space decomposition of a complex semisimple Lie algebra with respect to a Cartan subalgebra, which is a grading by a free abelian group (namely, the root lattice) and will be referred to as the Cartan grading (being unique up to equivalence of gradings --- see below). Recently, there has been a considerable interest in gradings on Lie algebras by arbitrary groups, motivated in part by Lie theory itself and in part by potential applications in mathematical physics (see our monograph \cite{EK_mon} and the references therein).

We recall that a \emph{grading} on a Lie algebra $\cL$ by a group $G$ is a vector space decomposition $\Gamma:\cL=\bigoplus_{g\in G}\cL_g$ such that $[\cL_g,\cL_h]\subseteq\cL_{gh}$ for all $g,h\in G$. The subspaces $\cL_g$ are referred to as the \emph{homogeneous components}, the elements of $\cL_g$ as \emph{homogeneous elements of degree $g$}, and the set $S=\{g\in G:\cL_g\ne 0\}$ as the \emph{support} of the grading (or of the $G$-graded algebra $\cL$). Clearly, these definitions make sense for arbitrary nonassociative (i.e., not necessarily associative) algebras. In the context of Lie algebras, it is natural to assume $G$ abelian, as in this case any $G$-grading on $\cL$ uniquely extends to the universal enveloping algebra $U(\cL)$. Also, if $\cL$ is simple then the elements of the support $S$ commute with one another. Therefore, throughout this article we will assume all grading groups to be abelian.

There are two natural equivalence relations on gradings, depending on whether the direct sum decomposition alone is important to us (so any group realizing it will do) or the group $G$ is important as well. Two gradings, $\cL=\bigoplus_{g\in G}\cL_g$ and $\cL=\bigoplus_{h\in H}\cL'_h$, are said to be \emph{equivalent} if there exists an automorphism $\vphi$ of $\cL$ and a bijection between the supports $\alpha:S\to S'$ such that $\vphi(\cL_g)=\cL'_{\alpha(g)}$ for all $g\in S$. Two $G$-gradings, $\cL=\bigoplus_{g\in G}\cL_g$ and $\cL=\bigoplus_{g\in G}\cL'_g$, are \emph{isomorphic} if there exists an automorphism $\vphi$ of $\cL$ such that $\vphi(\cL_g)=\cL'_g$ for all $g\in G$ (in other words, these gradings yield isomorphic $G$-graded algebras). 

Among all groups realizing a given grading on $\cL$, there is the universal one (unique up to isomorphism): it is generated by the support $S$ subject only to the relations $s_1s_2=s_3$ for all $s_1,s_2,s_3\in S$ such that $0\ne[\cL_{s_1},\cL_{s_2}]\subseteq\cL_{s_3}$. In the case of equivalent gradings, as above, the bijection $\alpha$ extends to a unique isomorphism of the universal grading groups.

The so-called \emph{fine gradings} (i.e., those that do not admit a proper refinement) are of fundamental importance. The Cartan grading mentioned above is one of them. The classification of fine gradings up to equivalence has recently been completed for all finite-dimensional simple Lie algebras over an algebraically closed field of characteristic $0$  --- see \cite[Chapters 3--6]{EK_mon} and also \cite{Yu_E,E_char0,DE_over}. In particular, there are $14$ fine gradings on each of the simple Lie algebras of types $E_6$, $E_7$ and $E_8$.

For a fixed abelian group $G$, the classification of $G$-gradings up to isomorphism is also known for the classical simple Lie algebras of all types except $E$ (over an algebraically closed field of characteristic different from $2$). The details can be found in \cite{BK_classical,EK_mon} and in \cite{EK_D4} for the special case $D_4$. For the purposes of this article, it is sufficient to note that every $G$-grading $\Gamma$ is a coarsening of at least one fine grading $\Delta$, so it is induced by a homomorphism $\nu:G^u\to G$, where $G^u$ is the universal group of $\Delta$. This means that $\Gamma$ is obtained by assigning the degree $\nu(s)\in G$ to all elements of $\cL$ that are homogeneous of degree $s\in G^u$ with respect to $\Delta$. Thus, we have a description of all $G$-gradings for the $E$ types as well (at least in characteristic $0$), although the isomorphism problem for these gradings remains open.

Let $\cL$ be a semisimple Lie algebra over an algebraically closed field $\FF$ of characteristic $0$. Suppose $\cL$ is graded by an abelian group $G$ (which can be assumed, without loss of generality, to be finitely generated) and let $V$ be an $\cL$-module. A \emph{compatible $G$-grading} on $V$ is a vector space decomposition $V=\bigoplus_{g\in G} V_g$ such that $\cL_g V_h\subseteq V_{gh}$ for all $g,h\in G$. If such a grading on $V$ is fixed, $V$ is said to be a \emph{graded $\cL$-module}. For example, if we take the Cartan grading on $\cL$ then every finite-dimensional module admits a compatible grading. Note that if a grading $V=\bigoplus_{g\in G} V_g$ is compatible then so is any of its \emph{shifts}, i.e., any grading obtained by picking $h\in G$ and declaring the elements of $V_g$ to be of degree $hg$, for all $g\in G$. 

Finite-dimensional $G$-graded $\cL$-modules for semisimple Lie algebras have been studied in \cite{EK_Israel} and \cite[Appendix~A]{EK_D4}. Every such module is a direct sum of \emph{graded simple} modules, i.e., $G$-graded modules that do not contain any proper nonzero graded submodule. Hence, if we know all finite-dimensional graded simple modules then we can also determine which finite-dimensional $\cL$-modules admit a compatible $G$-grading. 

\smallskip

\emph{Throughout the paper, all Lie algebras and modules considered will be assumed to be of finite dimension.}

\smallskip

Graded simple modules can be expressed in terms of simple $\cL$-modules by means of a suitable version of Clifford theory, as will be briefly explained in Section \ref{se:graded_modules} following \cite{EK_Israel,EK_D4} (see also \cite{MZ,EK_loop} for an alternative expression as loop modules, which works even in the infinite-dimensional case). For now, we only note that there are two obstacles for a simple $\cL$-module to admit a compatible $G$-grading. The first obstacle is related to the action of the outer automorphism group of $\cL$ on the isomorphism classes of simple modules. For the simple module $V(\lambda)$ of highest weight $\lambda$, this effect is measured by what we call the \emph{inertia group} of $\lambda$. The second obstacle is related to the fundamental group of the inner automorphism group of $\cL$ (in other words, the center of the corresponding simply connected group). It is measured by what we call the \emph{graded Brauer invariant} of $\lambda$ and, in particular, its degree, which we call the \emph{graded Schur index} of $\lambda$. These invariants have been explicitly computed in the above-mentioned works for all simple Lie algebras with the exception of the $E$ types. Since the automorphism group of the Lie algebra of type $E_8$ is connected and simply connected, there is nothing to do in this case: all finite-dimensional modules admit a compatible grading and graded simple modules are in fact simple. (The same, by the way, holds for $F_4$ and $G_2$.) Our present work is dedicated to the remaining two cases: $E_6$ and $E_7$. We will show that, for a half of the fine gradings on these Lie algebras (more precisely, $11$ out of $14$ for $E_6$ and $3$ out of $14$ for $E_7$), all graded Brauer invariants are trivial. (This can also be deduced from \cite{Yu_E}.) For the remaining fine gradings, some of the graded Brauer invariants are nontrivial, but their degree (i.e., the graded Schur index) turns out to be small: $2$ in the case of $E_7$ and $3$ in the case of $E_6$. It follows that, for any $G$-grading on the Lie algebra of type $E_7$, a graded simple module is either simple or direct sum of two copies of a simple module. For type $E_6$, we will show that a graded simple module can have at most three simple summands (more precisely, it is either simple or sum of two non-isomorphic simple modules that are dual to one another or sum of three copies of a simple module). Thus, in this context, the  behavior of the exceptional simple Lie algebras is remarkably different from the typical behavior in the series $A$, $B$, $C$ and $D$. Indeed, the results in \cite{EK_Israel,EK_D4} show that the maximum value of the graded Schur index is $r+1$ for type $A_r$, $2^r$ for type $B_r$, $2^{r-1}$ for type $D_r$ and $2^{k+1}$ for type $C_r$, where $k=\mathrm{ord}_2(r)$ (the greatest integer such that $2^k$ divides $r$).

The paper is organized as follows. After reviewing the background and introducing notation concerning graded modules in Section \ref{se:graded_modules}, we will establish some auxiliary results on algebraic groups in Section \ref{se:prelim}. In Section \ref{se:E6}, we consider the graded Brauer invariants for type $E_6$: we first show that they are all trivial in the case of the outer fine gradings (Theorem \ref{th:E6_outer}) and then explicitly compute them for the inner fine gradings (Theorem \ref{th:E6_inner}). Finally, in Section \ref{se:E7}, we compute the graded Brauer invariants for type $E_7$. In particular, we use the well known model of the Lie algebra of type $E_7$ in terms of $\Sl_8(\FF)$ (given e.g. in \cite{Adams}) and develop its version in terms of $\Sp_8(\FF)$, which may be of independent interest.


\section{Background on graded modules}\label{se:graded_modules}

Let $\cL$ be a semisimple Lie algebra over an algebraically closed field $\FF$ of characteristic $0$. These conditions on the ground field $\FF$ will be assumed throughout the paper without further mention. Any $G$-grading on $\cL$ corresponds to a homomorphism of algebraic groups $\wh{G}\rightarrow \Aut(\cL)$, $\chi\mapsto \alpha_\chi$, where $\wh{G}=\Hom(G,\FF^\times)$ is the group of characters (a quasitorus), so that the homogeneous components are (joint) eigenspaces for the automorphisms $\alpha_\chi$, i.e., $\cL_g=\{ x\in \cL:\alpha_\chi(x)=\chi(g)x\ \forall \chi\in \wh{G}\}$. 

Now we outline the theory of (finite-dimensional) graded $\cL$-modules following \cite{EK_Israel}. First of all, the group $\wh{G}$ permutes (isomorphism classes of) $\cL$-modules by \emph{twisting}: for $\chi\in\wh{G}$ and an $\cL$-module $V$, the twisted module $V^\chi$ coincides with $V$ as a vector space but has a different $\cL$-action: $x\cdot_\chi v=\alpha_\chi(x)v$ for all $x\in\cL$ and $v\in V$. It is easy to check that if $V$ admits a compatible $G$-grading then the linear map sending each $v\in V_g$ to $\chi(g)v$, for all $g\in G$, is an isomorphism of $\cL$-modules $V\to V^\chi$, so the twisted module $V^\chi$ is isomorphic to $V$ for any $\chi\in\wh{G}$.

Fix a Cartan subalgebra and a system of simple roots for $\cL$, and let $\Lambda^+$ denote the set of dominant integral weights. Since $\wh{G}$ acts on the isomorphism classes of irreducible $\cL$-modules, it also acts on $\Lambda^+$. (Explicitly, the image of $\alpha_\chi$ in the outer automorphism group of $\cL$ determines an automorphism of the Dynkin diagram and hence a permutation of dominant integral weights.) For any $\lambda \in\Lambda^+$, the \emph{inertia group} is 
\[
K_\lambda=\{\chi\in \wh{G}: \chi\ \text{fixes}\ \lambda\}.
\]
If $V(\lambda)$, the simple $\cL$-module with highest weight $\lambda$, admits a compatible $G$-grading then we have $K_\lambda=\wh{G}$.
In general, $K_\lambda$ is Zariski closed in $\wh{G}$ and $[\wh{G}:K_\lambda]$ is finite. Therefore, 
\[
H_\lambda\bydef (K_\lambda)^\perp=\{g\in G: \chi(g)=1\ \forall \chi\in K_\lambda\}
\]
is a finite subgroup of $G$ of size $\lvert H_\lambda\rvert=\lvert \wh{G}\lambda\rvert$, and $K_\lambda$ is canonically isomorphic to the group of characters of $G/H_\lambda$, so that an action of $K_\lambda$ corresponds to a $(G/H_\lambda)$-grading.

We will also need the $G$-graded Brauer group of $\FF$ (as defined in \cite{PP}), which we denote by $B_G(\FF)$. It consists of the equivalence classes of finite-dimensional central simple associative $\FF$-algebras equipped with a $G$-grading, with product induced by the tensor product over $\FF$. Since we assume $\FF$ algebraically closed, the group $B_G(\FF)$ has an easy interpretation: the elements of $B_G(\FF)$ can be represented by certain alternating bicharacters $\hat{\beta}\colon\wh{G}\times \wh{G}\to\FF^\times$ and the operation of $B_G(\FF)$ corresponds to the point-wise multiplication of bicharacters (see \cite[\S 2]{EK_Israel}). By analogy with linear algebra, we will say that a subgroup $Q\subset\wh{G}$ is \emph{isotropic} with respect to $\hat{\beta}$ if the restriction of $\hat{\beta}$ to $Q\times Q$ is trivial: $\hat{\beta}(Q,Q)=1$, and we define the \emph{radical} of $\hat{\beta}$ by $\rad\hat{\beta}=\{\chi\in\wh{G}:\hat{\beta}(\chi,\wh{G})=1\}$. For any $\hat{\beta}$ representing an element of $B_G(\FF)$, the radical has the form $T^\perp$ where $T$ is a finite subgroup of $G$, and $\hat{\beta}$ corresponds to a nondegenerate alternating bicharacter $\beta\colon T\times T\to\FF^\times$ by means of duality: $\beta(t_1,t_2)=\hat\beta(\chi_1,\chi_2)$ where $\chi_i$ is any character such that $\hat\beta(\psi,\chi_i)=\psi(t_i)$ for all $\psi\in\wh{G}$ ($i=1,2$). The pair $(T,\beta)$ carries the same information as $\hat{\beta}$, but this point of view has a disadvantage: one cannot define multiplication of such pairs directly. On the other hand, it has the advantage of an immediate connection with the original definition of the elements of $B_G(\FF)$ as equivalence classes: the unique graded division algebra $\cD$ in the Brauer class represented by $\hat{\beta}$ is isomorphic to the twisted group algebra $\FF^\sigma T$ (with its natural $T$-grading regarded as a $G$-grading), where $\sigma\colon T\times T\to\FF^\times$ is any $2$-cocycle satisfying $\beta(s,t)=\sigma(s,t)/\sigma(t,s)$ for all $s,t\in T$. Note that $T$ is the support of the grading on $\cD$, so we will refer to $T$ as the \emph{support of the Brauer class $[\cD]$}. As an algebra, $\cD$ is isomorphic to $M_\ell(\FF)$ where $|T|=\ell^2$; this $\ell$ will be referred to as the \emph{degree} of $[\cD]$.

For example, suppose $G$ contains a subgroup $T=\langle a,b\rangle\simeq\ZZ_\ell^2$. Pick a primitive $\ell$-th root of unity $\veps\in\FF$ and consider the following elements of $M_\ell(\FF)$:
\begin{equation}\label{Pauli}
X=\begin{bmatrix}
\veps^{\ell-1} & 0                 & 0           & \ldots      & 0     & 0\\
0           & \veps^{\ell-2}       & 0           & \ldots      & 0     & 0\\
\ldots      &                   &             &             &       &  \\[3pt]
0           & 0                 & 0           & \ldots      & \veps & 0\\
0           & 0                 & 0           & \ldots      & 0     & 1
\end{bmatrix}\mbox{ and }
Y=\begin{bmatrix}
0 & 1 & 0 & \ldots & 0 & 0\\
0 & 0 & 1 & \ldots & 0 & 0\\
\ldots & & & & \\[3pt]
0 & 0 & 0 & \ldots & 0 & 1\\
1 & 0 & 0 & \ldots & 0 & 0
\end{bmatrix},
\end{equation}
which we will call \emph{generalized Pauli matrices}. Clearly, $X^\ell=Y^\ell=I$ and $XY=\veps YX$, so the monomials $X^iY^j$, $0\le i,j<\ell$, constitute a basis of $M_\ell(\FF)$. Therefore, we can define a $G$-grading on $M_\ell(\FF)$, called a \emph{Pauli grading}, by assigning degree $a^ib^j$ to the monomial $X^iY^j$, which will be denoted by $X_{a^ib^j}$. Thus $M_\ell(\FF)$ becomes a graded division algebra with support $T$. Since $X_{a^ib^j}X_{a^{i'}b^{j'}}=\veps^{-i'j}X_{a^{i+i'}b^{j+j'}}$, the graded algebra $M_\ell(\FF)$ is isomorphic to $\FF^\sigma T$ where $\sigma(a^ib^j,a^{i'}b^{j'})=\veps^{-i'j}$, which gives $\beta(a^ib^j,a^{i'}b^{j'})=\veps^{ij'-i'j}$.

Now consider the simple $\cL$-module $V(\lambda)$ and the corresponding representation $\rho\colon U(\cL)\rightarrow \End_\FF(V(\lambda))$ of the universal enveloping algebra. As mentioned in the introduction, there may be two obstructions for $V(\lambda)$ to admit a compatible $G$-grading. The first and obvious one is the nontriviality of $H_\lambda$. The second is more subtle. There always exists a homomorphism of algebraic groups
\[
\begin{split}
K_\lambda&\longrightarrow \Aut\bigl(\End_\FF(V(\lambda))\bigr),\\
\chi&\mapsto \ \Int(u_\chi),
\end{split}
\]
where $\Int(u)$ denotes the inner automorphism $x\mapsto uxu^{-1}$ of $\End_\FF(V(\lambda))$, for $u\in\GL(V(\lambda))$, and $u_\chi$ is chosen to satisfy $\Int(u_\chi)\bigl(\rho(x)\bigr)=\rho\bigl(\alpha_\chi(x)\bigr)$ for all $x\in\cL$. This corresponds to a unique $(G/H_\lambda)$-grading on $\End_\FF(V(\lambda))$ such that $\rho$ becomes a homomorphism of $(G/H_\lambda)$-graded algebras. Here the $(G/H_\lambda)$-grading on $\cL$ is given by $\cL_{gH_\lambda}=\bigoplus_{h\in H_\lambda}\cL_{gh}$ for all $g\in G$ (the coarsening of the $G$-grading induced by the natural homomorphism $G\to G/H_\lambda$), and the grading on $U(\cL)$ is extended from $\cL$.
The class $[\End_\FF(V(\lambda))]$ in the $(G/H_\lambda)$-graded Brauer group $B_{G/H_\lambda}(\FF)$ is called the \emph{graded Brauer invariant} of $\lambda$, and the degree of the unique graded division algebra in this class is called the \emph{graded Schur index} of $\lambda$. It is the smallest natural number $k$ such that the direct sum $V(\lambda)^k$ of $k$ copies of $V(\lambda)$ admits a compatible $(G/H_\lambda)$-grading. 
The alternating bicharacter $\hat{\beta}_\lambda\colon K_\lambda\times K_\lambda\to\FF^\times$ representing the graded Brauer invariant of $\lambda$ is given by
\[
\hat{\beta}_\lambda(\chi_1,\chi_2)=u_{\chi_1}u_{\chi_2}u_{\chi_1}^{-1}u_{\chi_2}^{-1}\defby [u_{\chi_1},u_{\chi_2}]
\quad\forall \chi_1,\chi_2\in K_\lambda.
\]
(Note that, since $\Int(u_{\chi_1})$ and $\Int(u_{\chi_2})$ commute, $u_{\chi_1}$ and $u_{\chi_2}$ commute up to a scalar, which depends on $\chi_1$ and $\chi_2$ but not on the choice of $u_{\chi_1}$ and $u_{\chi_2}$.)

For each $\wh{G}$-orbit $\cO$ in $\Lambda^+$, select a representative $\lambda$. Let $k$ be the graded Schur index of $\lambda$, so that $V(\lambda)^k$ has a compatible $(G/H_\lambda)$-grading (unique up to isomorphism and shift). Finally, let $W(\cO)$ be the induced module:
\[
W(\cO)= \Ind_{K_\lambda}^{\wh G}\bigl(V(\lambda)^k\bigr)\bydef \FF\wh{G}\otimes_{\FF K_\lambda}V(\lambda)^k.
\]
(Here $\FF \wh{G}$ denotes the group algebra of $\wh{G}$, and similarly for $\FF K_\lambda$.) Then, by \cite[Theorem 8]{EK_Israel}, up to isomorphism and shift of the grading, these $W(\cO)$ are precisely the (finite-dimensional) $G$-graded simple $\cL$-modules.

In this way, the graded simple modules are obtained explicitly in terms of simple modules. The inertia groups, graded Brauer invariants and Schur indices have been computed for the classical simple Lie algebras in \cite{EK_Israel} and \cite[Appendix~A]{EK_D4}. In the latter, an interpretation of graded Brauer invariants in terms of algebraic groups is given as follows. 

Denote by $\cG$ any semisimple algebraic group with $\Lie(\cG)=\cL$. Among these, we denote by $\bar{\cG}$ the adjoint group and by $\tilde{\cG}$ the simply connected group, so we have quotient maps $\tilde{\cG}\to\cG\to\bar{\cG}$. It is well known that $Z(\tilde{\cG})$, the kernel of $\tilde{\cG}\to\bar{\cG}$, is isomorphic to the group of characters of the finite abelian group $\Lambda/\Lambda^\mathrm{r}$, where $\Lambda$ is the weight lattice and $\Lambda^\mathrm{r}$ is the root lattice of $\cL$. Recall also that $\Aut(\cL)$ is isomorphic to the semidirect product $\bar\cG\rtimes\Aut(\mathrm{Dyn})$, where $\mathrm{Dyn}$ is the Dynkin diagram of $\cL$. Let $S_\lambda$ be the stabilizer of $\lambda$ in $\Aut(\mathrm{Dyn})$. Then the inertia group $K_\lambda$ is the inverse image of the subgroup $\bar{\cG}\rtimes S_\lambda$ of $\bar\cG\rtimes\Aut(\mathrm{Dyn})\simeq\Aut(\cL)$ under the homomorphism $\chi\mapsto\alpha_\chi$. The representation $\rho\colon\cL\to\Gl(V(\lambda))$ ``integrates'' to $\tilde{\cG}$ and then extends to $\tilde{\rho}\colon\tilde{\cG}\rtimes S_\lambda\to\GL(V(\lambda))$.

Let $\pi\colon\tilde\cG\rtimes\Aut(\mathrm{Dyn})\rightarrow \bar\cG\rtimes\Aut(\mathrm{Dyn})$ be the quotient map, and consider preimages $\tilde\alpha_\chi$ in $\tilde\cG\rtimes\Aut(\mathrm{Dyn})$ of $\alpha_\chi$ for all $\chi\in\widehat{G}$. Since the commutator $[\alpha_{\chi_1},\alpha_{\chi_2}]$ is trivial, we have $[\tilde\alpha_{\chi_1},\tilde\alpha_{\chi_2}]\in Z(\tilde\cG)$ for all $\chi_1,\chi_2\in\widehat{G}$.
Since the action of the algebraic group $\tilde\cG\rtimes\Aut(\mathrm{Dyn})$ on $\cL$ is the adjoint action, we have 
$\rho\bigl(\alpha_\chi(x)\bigr)=\tilde\rho(\tilde\alpha_\chi)\rho(x)\tilde\rho(\tilde\alpha_\chi)^{-1}$,
for any $\chi\in K_\lambda$ and $x\in\cL$, so we can choose $u_\chi=\tilde{\rho}(\tilde{\alpha}_\chi)$. 
Finally, the elements in the center $Z(\tilde\cG)$ act by scalars on $V(\lambda)$, so there is an associated homomorphism 
\[
\Psi_\lambda\colon Z(\tilde\cG)\rightarrow \FF^\times.
\] 
As the result of this discussion, we obtain
\begin{equation}\label{eq:new}
\hat{\beta}_\lambda(\chi_1,\chi_2)=\Psi_\lambda\bigl([\tilde\alpha_{\chi_1},\tilde\alpha_{\chi_2}]\bigr).
\end{equation}
In particular, if $\lambda\in\Lambda^\mathrm{r}$ then $\Psi_\lambda$ is trivial and hence the graded Brauer invariant of $\lambda$ is trivial, too. Therefore, if the group $\Aut(\cL)$ is connected and simply connected then all finite-dimensional $\cL$-modules admit compatible $G$-gradings and all (finite-dimensional) graded simple  modules are simple. It was pointed out in \cite{EK_D4} that this is the case for the exceptional simple Lie algebras $G_2$, $F_4$ and $E_8$.

As stated in the introduction, the aim of this paper is to compute the graded Brauer invariants for $G$-gradings on the remaining simple Lie algebras $E_6$ and $E_7$. Each of these algebras has $14$ maximal quasitori (up to conjugation) in its automorphism group, which give rise to $14$ fine gradings (up to equivalence). As already mentioned, every $G$-grading is obtained by means of at least one homomorphism $G^u\to G$ where $G^u$ is the universal group of a fine grading, which can be identified with the group of characters of the corresponding maximal quasitorus $Q$. Upon this identification, the dual homomorphism $\wh{G}\to Q$ is precisely the map $\chi\mapsto\alpha_{\chi}$, so it sends $K_\lambda$ to $Q_\lambda\bydef Q\cap(\bar{\cG}\rtimes S_\lambda)$. From Equation \eqref{eq:new}, we see that the alternating bicharacter $\hat{\beta}_\lambda\colon K_\lambda\times K_\lambda\to\FF^\times$ representing the graded Brauer invariant of a given $\lambda\in\Lambda^+$ with respect to the $G$-grading is just the composition of the map $K_\lambda\times K_\lambda\to Q_\lambda\times Q_\lambda$ with the alternating bicharacter 
\begin{equation}\label{eq:bichar_for_Q}
\hat{\beta}_\lambda^Q\colon Q_\lambda\times Q_\lambda\to\FF^\times,\quad (x,y)\mapsto \Psi_\lambda\bigl([\tilde{x},\tilde{y}]\bigr),
\end{equation}
where $\tilde x$ and $\tilde y$ are arbitrary preimages of $x$ and $y$. This latter bicharacter $\hat{\beta}_\lambda^Q$ represents the graded Brauer invariant of the same $\lambda$, but with respect to the $G^u$-grading on $\cL$. 
Therefore, it suffices to compute the Brauer invariants for fine gradings.


\section{Preliminaries on algebraic groups}\label{se:prelim}

First of all, we observe that, for our purposes, it will be sufficient to deal with quasitori in connected algebraic groups. Indeed, in the case of $E_7$, the Dynkin diagram has trivial automorphism group, while in the case of $E_6$, this automorphism group has order $2$ and acts nontrivially on $\Lambda/\Lambda^\mathrm{r}\simeq\ZZ_3$, so the only $\lambda\in\Lambda^+$ with nontrivial stabilizer $S_\lambda$ are those in $\Lambda^\mathrm{r}$, which have trivial graded Brauer invariants. For the remaining $\lambda$, the quasitorus $Q_\lambda$ coincides with $Q_\mathrm{in}\bydef Q\cap\inaut(\cL)$ and thus is contained in the connected group $\inaut(\cL)$. The gradings whose quasitori are contained in $\inaut(\cL)$ will be referred to as \emph{inner} and all other gradings as \emph{outer}.

We will make use of the following results on semisimple algebraic groups. As before, let $\tilde{\cG}$ be the simply connected group corresponding to a semisimple algebraic group $\cG$, let $\pi\colon\tilde{\cG}\to\cG$ be the canonical projection and let $Z$ be the kernel of $\pi$, which is a subgroup of $Z(\tilde{\cG})$. (We will be primarily interested in the case of the adjoint group, where $Z=Z(\tilde{\cG})$.) For a quasitorus $P\subset\cG$, we define an alternating bihomomorpism
\begin{equation}\label{eq:def_beta_quasitorus}
\hat{\beta}^P\colon P\times P\to Z,\quad (x,y)\mapsto [\tilde{x},\tilde{y}],
\end{equation}
where $\tilde{x}$ and $\tilde{y}$ are preimages of $x$ and $y$ under $\pi$. Note that, if $Q_\lambda\subset\inaut(\cL)$  (in particular, if $S_\lambda$ is trivial) then $\hat{\beta}_\lambda^Q$ in Equation \eqref{eq:bichar_for_Q} is just the composition of $\hat{\beta}^{Q_\lambda}$ and $\Psi_\lambda$. By analogy with bicharacters, we will speak about the radical of $\hat{\beta}^P$ and isotropic subgroups of $P$ with respect to $\hat{\beta}^P$.

For a prime $p$, by the $p$-{\em torsion} of an abelian group we will mean the subgroup consisting of all $p$-elements (i.e., the elements whose order is a power of $p$).

\begin{lemma}\label{lm:1}
Let $P\subset\cG$ be a quasitorus, so $P=P_0\times P_1$ where $P_0$ is a torus (the connected component of $P$) and $P_1$ is a finite abelian group. Write $P_1=P_1'\times P_1''$ where $P_1'$ is the product of the $p$-torsion subgroups of $P_1$ for all primes $p$ dividing $|Z|$. Then $\hat{\beta}^P(P,P_0\times P_1'')=1$ and $\pi^{-1}(P_0\times P_1'')$ is a quasitorus.
\end{lemma}

\begin{proof}
Since $\pi^{-1}(P_0)$ is a quasitorus (being a subgroup of the maximal torus of $\tilde\cG$ that projects onto a maximal torus of $\cG$ containing $P_0$), we have $\pi^{-1}(P_0)=ZS$ where $S$ is a torus that projects onto $P_0$. Since the commutator is a continuous map $\pi^{-1}(P)\times\pi^{-1}(P)\to Z$, $S$ is connected and $Z$ is finite, we conclude that $S$ commutes with all elements of $\pi^{-1}(P)$ and hence $\hat{\beta}^P(P,P_0)=1$. 

Now consider $P_1''$. Since $|P_1''|$ and $|Z|$ are coprime, we have $\pi^{-1}(P_1'')=Z\times\tilde{P}_1''$ where $\tilde{P}_1''\simeq P_1''$. Since $\tilde{P}_1''$ is a characteristic subgroup of the normal subgroup $\pi^{-1}(P_1'')$ of $\pi^{-1}(P)$, it is itself normal in $\pi^{-1}(P)$. But this implies $[\pi^{-1}(P),\tilde{P}_1'']\subset\tilde{P}_1''\cap Z=1$ and hence $\hat{\beta}^P(P,P_1'')=1$.
\end{proof}

\begin{corollary}\label{cor:Brauer_E6E7}
Let $\Gamma$ be a $G$-grading on the simple Lie algebra $\cL$ of type $E_6$, respectively $E_7$. Then the support of the graded Brauer invariant of any $\lambda\in\Lambda^+$ is contained in the $3$-torsion, respectively $2$-torsion, of $G/H_\lambda$.
\end{corollary}

\begin{proof}
Without loss of generality, we may assume that $G$ is generated by the support of $\Gamma$, so the homomorphism  $\wh{G}\to\Aut(\cL), \chi\mapsto\alpha_\chi,$ is an embedding. Let $Q\subset\Aut(\cL)$ be the image of this homomorphism, so $Q\simeq\wh{G}$. As noted above, in the case of $E_6$, we may assume that $S_\lambda$ is trivial, so the graded Brauer invariant of $\lambda$ is associated to the coarsening of $\Gamma$ (a $(G/H_\lambda)$-grading) determined by the quasitorus $Q_\lambda=Q_\mathrm{in}$. Replacing $G$ with $G/H_\lambda$, we may assume that $Q\subset\inaut(\cL)$. 
Let $p=|Z(\tilde{\cG})|$, i.e., $p=3$ in the case of $E_6$ and $p=2$ in the case of $E_7$. Write $G=G'\times G''$ where $G'$ is the $p$-torsion of $G$. Accordingly, we have $Q=Q'\times Q''$ where $Q''=(G')^\perp\simeq\wh{G''}$ and $Q'=(G'')^\perp\simeq\wh{G'}$. By Lemma \ref{lm:1}, $\hat{\beta}^Q(Q,Q'')=1$ and hence $Q''$ is in the radical of $\hat{\beta}^Q_\lambda$, so the support of the Brauer invariant of $\lambda$ is contained in $(Q'')^\perp=G'$. 
\end{proof}

We will say that two elements $x,y\in\cG$ {\em strongly commute in $\cG$} if their preimages commute in $\tilde\cG$. Thus, if $P\subset\cG$ is a quasitorus and $x,y\in P$ then $\hat{\beta}^P(x,y)=1$ if and only if $x$ and $y$ strongly commute in $\cG$.

\begin{lemma}\label{lm:2}
Suppose $\cH\subset\cG$ are semisimple algebraic groups and $x,y\in\cH$. If $x$ and $y$ strongly commute in $\cH$ then they also strongly commute in $\cG$.
\end{lemma}

\begin{proof}
Since $\tilde{\cH}$ is simply connected, there exists a homomorphism $f\colon\tilde{\cH}\to\tilde{\cG}$ that makes the following diagram commute:
\[
\xymatrix{
\tilde{\cH}\,\ar[r]^{f}\ar@{->>}[d] & \,\tilde{\cG}\ar@{->>}[d]\\
\cH\,\ar@{^{(}->}[r] & \,\cG
}
\]
Since preimages $\tilde{x}$ and $\tilde{y}$ of our elements in $\tilde{\cH}$ commute with one another, so do their preimages $f(\tilde{x})$ and $f(\tilde{y})$ in $\tilde{\cG}$.
\end{proof}

\begin{remark}
The converse of Lemma \ref{lm:2} does not hold. For example, for $\cL$ of type $E_6$ and $Q_1$ as in Remark \ref{re:E6gKO} below, $Q_1=\langle \sigma_1,\sigma_2\rangle$ where $\sigma_1$ and $\sigma_2$ strongly commute in $\Aut(\cL)$, but not in $\Aut(\cO)\simeq\Aut\bigl(M_3(\FF)\bigr)\simeq\SL_3(\FF)/\langle\omega I\rangle\leq \Aut(\cL)$. Also, for $\cL$ of type $E_7$,
the elements $[I_1\otimes X_2]$ and $[I_1\otimes Y_2]$ in \eqref{eq:kerbZ42Z23} strongly commute in $\Aut(\cL)$ but not in $\PGL(V_2)=\SL(V_2)/\langle -I\rangle\leq \SL(V)/\langle\bi I\rangle\leq \Aut(\cL)$.
\end{remark}

The following result is well known, but we sketch the proof for completeness.

\begin{lemma}\label{lm:3}
Let $x$ and $y$ be commuting semisimple elements in a semisimple algebraic group $\cG$. Then the following conditions are equivalent:
\begin{enumerate}
\item $x$ and $y$ strongly commute in $\cG$;
\item $y$ lies in the connected component of the centralizer $\Centr_\cG(x)$;
\item $x$ and $y$ are contained in a torus in $\cG$;
\item the fixed subalgebra $\Fix_{\Lie(\cG)}(x,y)$ contains a Cartan subalgebra of $\Lie(\cG)$.
\end{enumerate}
\end{lemma}

\begin{proof}
(1)$\Rightarrow$(2) By a theorem of Steinberg (see e.g. \cite[Th. 3.5.6]{Carter}), if $G$ is a connected reductive algebraic group such that $[G,G]$ is simply connected then the centralizers of semisimple elements in $G$ are connected. In particular, this applies to simply connected semisimple groups. If $x$ and $y$ strongly commute in $\cG$ then $\tilde{y}\in\Centr_{\tilde{\cG}}(\tilde{x})$. Since this centralizer is connected, $\pi$ maps it to the connected component of $\Centr_{\cG}(x)$.

(2)$\Rightarrow$(3) Since $y$ is a semisimple element of the connected component of $\Centr_\cG(x)$, it is contained in a maximal torus of $\Centr_\cG(x)$, which also contains $x$, because $x$ is in the center of $\Centr_\cG(x)$.

(3)$\Rightarrow$(1) Let $\cT$ be a maximal torus of $\cG$ containing $x$ and $y$. Then $\tilde{\cT}=\pi^{-1}(\cT)$ is a maximal torus of $\tilde{\cG}$, hence $\tilde{x}$ and $\tilde{y}$ commute.

(3)$\Leftrightarrow$(4) Any Cartan subalgebra of $\Lie(\cG)$ has the form $\Lie(\cT)$ for some maximal torus $\cT$ of $\cG$. It remains to observe that $x$ and $y$ fix $\Lie(\cT)$ point-wise if and only if $\cT$ centralizes $x$ and $y$, and the latter condition is equivalent to $x,y\in\cT$ since $\cT$ is self-centralizing in $\cG$.
\end{proof}

\begin{remark}
The equivalence of the first three conditions is valid in arbitrary characteristic.
\end{remark}

Subgroups of $\cG$ that are contained in a torus are said to be \emph{toral} in $\cG$.


\section{Graded Brauer invariants for $E_6$}\label{se:E6}

Let $\cL$ be the simple Lie algebra of type $E_6$. As mentioned in the introduction, up to equivalence, $\cL$ has $14$ fine gradings \cite{DV_E6}, which are listed in \cite[Fig. 6.2]{EK_mon}.  There are $5$ fine gradings that are inner, i.e., the corresponding maximal quasitorus $Q$ is contained in $\inaut(\cL)$; their universal groups are $\ZZ^6$ (the Cartan grading), $\ZZ^2\times\ZZ_3^2$, $\ZZ^2\times\ZZ_2^3$, $\ZZ_3^4$ and $\ZZ_2^3\times\ZZ_3^2$. The remaining $9$ are  outer, i.e., $Q$ is not contained in $\inaut(\cL)$ and hence $Q_\mathrm{in}=Q\cap\inaut(\cL)$ has index $2$ in $Q$. It follows that in each of these outer cases the  universal group $G^u$ contains a {\em distinguished element} $h$ of order $2$ that is characterized by the property that $Q_\mathrm{in}=\langle h\rangle^\perp$ or, equivalently, the coarsening induced by the natural homomorphism $G^u\to G^u/\langle h\rangle$ is an inner grading. The universal groups of the outer fine gradings are $\ZZ^4\times\ZZ_2$, $\ZZ^2\times\ZZ_2^3$, $\ZZ\times\ZZ_2^5$, $\ZZ\times\ZZ_2^4$, $\ZZ_2\times\ZZ_3^3$, $\ZZ_2^7$, $\ZZ_2^6$, $\ZZ_4^3$ and $\ZZ_4\times\ZZ_2^4$. It so happens that all $14$ universal groups are distinct, with the exception of $\ZZ^2\times\ZZ_2^3$, but in this case one of the gradings is inner and the other outer, so we may refer to each fine grading using its universal group and, if necessary, indicating whether it is inner or outer.

\subsection{Outer fine gradings}

\begin{theorem}\label{th:E6_outer}
For each of the $9$ outer fine gradings of the simple Lie algebra of type $E_6$, the graded Brauer invariants of all $\lambda\in\Lambda^+$ are trivial.
\end{theorem}

\begin{proof}
As already pointed out in Section \ref{se:prelim}, if $\lambda\in\Lambda^\mathrm{r}$ then the Brauer invariant is trivial, and otherwise $Q_\lambda=Q_\mathrm{in}$ and hence $H_\lambda=\langle h\rangle$. Looking at the $9$ universal groups, we see that, in $8$ cases, the group  $G^u/\langle h\rangle$ has no $3$-torsion, so the Brauer invariants are trivial by Corollary \ref{cor:Brauer_E6E7}. 

In the remaining case $G^u=\ZZ_2\times\ZZ_3^3$, the grading can be constructed as follows (see, for instance, \cite[\S 6.2]{EK_mon}). We realize $\cL$ as the Lie subalgebra of $\frgl(\cA)$, where $\cA$ is the so-called Albert algebra (the simple exceptional Jordan algebra of dimension $27$), given by $\cL=\Der(\cA)\oplus L_{\cA_0}$ (a grading by $\ZZ_2$), where $\cA_0$ is the subspace of trace zero elements in $\cA$, and $L_{\cA_0}=\espan{L_x:x\in\cA_0}$, with $L_x$ being the operator of multiplication by $x$ in $\cA$. The $\ZZ_3^3$-grading on $\cA$ (see e.g. \cite[\S 5.2]{EK_mon}) then induces the desired $\ZZ_2\times\ZZ_3^3$-grading on $\cL$.
There is only one element $h$ of order $2$ in $G^u$, and the resulting grading by $G^u/\langle h\rangle$ is obtained by using the $\ZZ_3^3$-grading on $\cA$ and forgetting about the $\ZZ_2$-grading. Recall from Equations \eqref{eq:bichar_for_Q} and \eqref{eq:def_beta_quasitorus} that $\hat{\beta}^Q_\lambda(x,y)=\Psi_\lambda(\hat{\beta}^{Q_\lambda}(x,y))$, so it suffices to show that $\hat{\beta}^{Q_\mathrm{in}}$ is trivial, i.e., any elements $x,y\in Q_\mathrm{in}$ strongly commute in $\Aut(\cL)$. There is a natural embedding of $\Aut(\cA)$, which is a simple algebraic group of type $F_4$, into $\Aut(\cL)$. Since the quasitorus $Q_\mathrm{in}$ is contained in the image of $\Aut(\cA)$, and $\Aut(\cA)$ is simply connected, Lemma \ref{lm:2} completes the proof. 
\end{proof}

\begin{corollary}
Let $\cL$ be the simple Lie algebra of type $E_6$ and let $\sigma$ be the nontrivial automorphism of its Dynkin diagram. Suppose $\cL$ is equipped with a $G$-grading $\Gamma$ that is induced from one of the outer fine gradings by a homomorphism $\nu\colon G^u\to G$ where $G^u$ is the universal group of the fine grading. Let $h\in G^u$ be the distinguished element.
\begin{enumerate}
\item[(a)] If $\nu(h)\ne e$ (i.e., if $\Gamma$ is outer) then a finite-dimensional $\cL$-module $V$ admits a compatible $G$-grading if and only if, for any $\lambda\in\Lambda^+$, the simple $\cL$-modules $V(\lambda)$ and $V(\sigma(\lambda))$ occur in $V$ with the same multiplicity.
\item[(b)] If $\nu(h)=e$ (i.e., if $\Gamma$ is inner) then every finite-dimensional $\cL$-module admits a compatible $G$-grading. 
\end{enumerate}
\end{corollary}

\begin{proof}
Recall that the graded Brauer invariant of $\lambda$ with respect to $\Gamma$ is determined by $\hat{\beta}_\lambda(\chi_1,\chi_2)=\hat{\beta}^Q_\lambda(\alpha_{\chi_1},\alpha_{\chi_2})$, for $\chi_1,\chi_2\in K_\lambda$, so Theorem \ref{th:E6_outer} implies that $\hat{\beta}_\lambda$ is trivial, which means that $V(\lambda)$ admits a compatible $(G/H_\lambda)$-grading. Hence, there are two kinds of $G$-graded simple $\cL$-modules: $V(\lambda)$ for $\lambda$ stabilized by $\wh{G}$ and $\Ind_{K_\lambda}^{\wh{G}} V(\lambda)$ for all other $\lambda$. The result follows.
\end{proof}

\subsection{Inner fine gradings} 
As already mentioned, the universal groups are $\ZZ^6$, $\ZZ^2\times\ZZ_2^3$, $\ZZ^2\times\ZZ_3^2$, $\ZZ_2^3\times \ZZ_3^2$ and $\ZZ_3^4$, which we divide into three cases.

\medskip

$\boxed{\ZZ^6\text{ and }\ZZ^2\times\ZZ_2^3}$\quad
Since there is no $3$-torsion, the Brauer invariants are trivial.

\medskip

$\boxed{\ZZ^2\times\ZZ_3^2\text{ and }\ZZ_2^3\times\ZZ_3^2}$\quad
To construct these gradings, we think of $\cL$ in terms of Tits construction $\cT(\cC,\cJ)$, where $\cC$ is the Cayley algebra and $\cJ$ is the Jordan algebra $M_3(\FF)^{(+)}$, i.e., the algebra of $3\times 3$ matrices with the symmetrized product: $x\cdot y=\frac{1}{2}(xy+yx)$. Recall that $\cT(\cC,\cJ)$ is defined as the vector space direct sum
\[
\cT(\cC,\cJ)=\Der(\cC)\oplus(\cC_0\otimes\cJ_0)\oplus\Der(\cJ),
\]
where the subscript $0$ refers to the subspace of trace zero elements. Here $\Der(\cC)\oplus\Der(\cJ)$ is a subalgebra of $\cT(\cC,\cJ)$, the middle term $\cC_0\otimes\cJ_0$ is its natural module, and only the definition of the bracket of two elements in this middle term requires some care (see e.g. \cite[\S 6.2]{EK_mon}). In particular, the group $\Aut(\cC)\times\Aut(\cJ)$ is naturally embedded in $\Aut(\cL)$. It is well known that $\Aut(\cC)$ is the simple algebraic group of type $G_2$ and $\Aut(\cJ)$ contains the simple algebraic group $\Aut(M_3(\FF))\simeq\PSL_3(\FF)$ of type $A_2$ as a subgroup of index $2$ (the connected component).

Recall the $\ZZ_3^2$-grading on the matrix algebra $M_3(\FF)$ associated to the generalized Pauli matrices $X,Y\in\SL_3(\FF)$ (see Equation \eqref{Pauli}):  $X^3=Y^3=I$ and $XY=\omega YX$, where $\omega\in\FF$ is a primitive cubic root of unity, and $M_3(\FF)_{(\bar{\imath},\bar{\jmath})}=\FF X^iY^j$, for any $0\leq i,j<3$. Then the corresponding quasitorus $Q_2=\langle [X],[Y]\rangle$ in $\Aut(M_3(\FF))$, where brackets denote the class of an element of $\SL_3(\FF)$ modulo the center, is maximal, and it remains maximal in $\Aut(\cJ)$ \cite[\S 5.6]{EK_mon}.

The two gradings in question are obtained by combining this $\ZZ_3^2$-grading on $\cJ$ with each of the two possible fine gradings (up to equivalence) on $\cC$, whose universal groups are $\ZZ^2$ and $\ZZ_2^3$ (see e.g. \cite[\S 4.2]{EK_mon}). Thus, if $Q_1$ is the maximal quasitorus of $\Aut(\cC)$ corresponding to one of these two fine gradings on $\cC$ then $Q=Q_1\times Q_2\leq \Aut(\cC)\times\Aut(\cJ)\leq\Aut(\cL)$ is the maximal quasitorus of $\Aut(\cL)$ giving us a desired fine grading on $\cL$.  The rank of $\Fix_{\cL}(Q_2)=\Der(\cC)$ is $2$, so $[X]$ and $[Y]$ do not strongly commute in $\Aut(\cL)$ by Lemma \ref{lm:3}. Therefore, the restriction of $\hat{\beta}^Q$ to $Q_2$ is nontrivial, so Lemma \ref{lm:1} implies that the radical of $\hat{\beta}^Q$ is equal to $Q_1$. More precisely, the homomorphism $f$ that completes the commutative diagram
\[
\xymatrix{
\SL_3(\FF)\,\ar[r]^{f}\ar@{->>}[d] & \,\tilde{\cG}\ar@{->>}[d]\\
\Aut(M_3(\FF))\,\ar@{^{(}->}[r] & \,\inaut(\cL)
}
\]
where $\tilde{\cG}$ is the simply connected group of type $E_6$, is an embedding, and we have $\hat{\beta}^Q([X],[Y])=f(\omega I)$.

\medskip

$\boxed{\ZZ_3^4}$\quad The maximal quasitorus of $\Aut(\cL)$ corresponding to this grading is fully described in \cite{DV_E6}. It is generated by four order $3$ automorphisms:
\[
Q=\langle F_1,F_2,F_3,F_4\rangle,
\]
with $\langle F_1,F_2\rangle$ nontoral \cite[Lemma 9]{DV_E6}, $\langle F_2,F_3,F_4\rangle$ toral \cite[Remark 4]{DV_E6} (hence isotropic for $\hat\beta^Q$ by Lemma \ref{lm:3}), and $\langle F_1,F_3,F_4\rangle$ nontoral but with any two elements generating a toral subgroup \cite[Remark 5]{DV_E6} (hence also isotropic for $\hat\beta^Q$). This proves at once that the radical of $\hat\beta^Q$ is $\langle F_3,F_4\rangle$.

\begin{remark}\label{re:E6gKO}
This fine $\ZZ_3^4$-grading on the simple Lie algebra $\cL$ of type $E_6$ can also be described in terms of the symmetric construction of Freudenthal's Magic Square $\frg(\cC_1,\cC_2)$ where $\cC_1$ and $\cC_2$ are symmetric composition algebras. (The reader is referred to \cite[\S 6.2]{EK_mon} for background, details, and references.) Let $\bar\cK$ be the para-Hurwitz algebra associated to the quadratic algebra $\cK=\FF\times\FF$ and let $\cO$ be the Okubo algebra over $\FF$. Then $\cL$ is isomorphic to the Lie algebra
\[
\frg(\cO,\bar\cK)\bydef \bigl(\tri(\cO)\oplus\tri(\bar\cK)\bigr)\oplus\Bigl(\bigoplus_{i=1}^3\iota_i(\cO\otimes\bar\cK)\Bigr),
\]
the direct sum of the \emph{triality Lie algebras} of $\cO$ and $\bar\cK$, respectively, and of three copies of $\cO\otimes\bar\cK$. There is a natural order $3$ automorphism $\Theta$ of $\frg(\cO,\bar\cK)$ (see \cite[Equation (6.23)]{EK_mon}) that permutes cyclically these three copies of $\cO\otimes\bar\cK$. Also, there is a natural embedding
\[
\Aut(\cO)\times\Aut(\bar\cK)\hookrightarrow \Aut\bigl(\frg(\cO,\bar\cK)\bigr).
\]
The Okubo algebra has a fine $\ZZ_3^2$-grading (see e.g. \cite[\S 4.6]{EK_mon}). Let $Q_1\leq\Aut(\cO)$ be the corresponding maximal quasitorus. The group $\Aut(\bar\cK)$ is the symmetric group of order $3$. If $\vartheta$ is one of its order $3$ elements, then $Q_2\bydef\langle \Theta,\vartheta\rangle$ is a nontoral subgroup of $\Aut\bigl(\frg(\cO,\bar\cK)\bigr)$ isomorphic to $\ZZ_3^2$, whose centralizer is 
$\Centr(Q_2)=\Aut(\cO)\times Q_2$ (\cite[Proposition 6.38 and Corollary 6.39]{EK_mon}). Hence $Q=Q_1\times Q_2$ is a maximal quasitorus of $\Aut\bigl(\frg(\cO,\bar\cK)\bigr)$, and it gives rise to the desired fine $\ZZ_3^4$-grading on $\frg(\cO,\bar\cK)$.
\end{remark}

Using this model, we can determine the radical of $\hat{\beta}^Q$ independently of \cite{DV_E6}. Indeed, the subgroup $Q_1\times\langle\vartheta\rangle$ is toral, because its fixed subalgebra is the direct sum of a Cartan subalgebra of $\tri(\cO)\simeq\frso(\cO)$ and the two-dimensional abelian subalgebra $\tri(\bar\cK)$, and this sum is a Cartan subalgebra of $\frg(\cO,\bar\cK)$. Therefore, $Q_1\times\langle\vartheta\rangle$ is isotropic for $\hat\beta^Q$. Also, for $\Theta$, we have  
\[
\Fix_{\frg(\cO,\bar\cK)}(\Theta)=\Der(\cO)\oplus\espan{\sum_{i=1}^3\iota_i(x\otimes y):  x\in\cO,\,y\in\bar\cK},
\]
and, for any $1\ne \sigma\in Q_1$, $\Fix_{\cO}(\sigma)$ is a Cartan subalgebra of the Lie algebra $\cO^{(-)}\simeq \frsl_3(\FF)$, hence
\[
\Fix_{\frg(\cO,\bar\cK)}(\Theta,\sigma)=\ad\bigl(\Fix_{\cO}(\sigma)\bigr)\oplus\espan{\sum_{i=1}^3 \iota_i(x\otimes y): x\in \Fix_{\cO}(\sigma),\,y\in\bar\cK}
\]
is a Cartan subalgebra of $\frg(\cO,\bar\cK)$. Therefore, for any $\sigma\in Q_1$, the subgroup $\langle\sigma,\Theta\rangle$ is toral and hence isotropic for $\hat\beta^Q$. We conclude that the radical of $\hat\beta^Q$ is $Q_1$.

\medskip

Now we are going to summarize our results concerning inner fine gradings for $E_6$. We will fix the following enumeration of simple roots $\{\alpha_1,\ldots,\alpha_6\}$ and the corresponding fundamental weights $\{\fw_1,\ldots,\fw_6\}$:

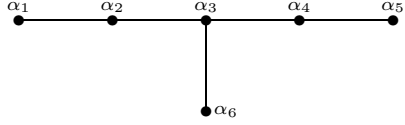
\[
\xymatrix{
\bullet\ar@{-}[]+0;[r]+0^<{\alpha_1}  & \bullet\ar@{-}[]+0;[r]+0^<{\alpha_2} & \bullet\ar@{-}[]+0;[r]+0^<{\alpha_3}\ar@{-}[]+0;[d]+0^>{\alpha_6}  & \bullet\ar@{-}[]+0;[r]+0^<{\alpha_4} & \bullet\ar@{}[]+0;[r]^<{\alpha_5} & \\
& & \bullet & &
}
\]

\begin{theorem}\label{th:E6_inner}
Let $\cL$ be the simple Lie algebra of type $E_6$. For the inner fine gradings on $\cL$ with universal groups $\ZZ^6$ and $\ZZ^2\times\ZZ_2^3$, the graded Brauer invariants of all $\lambda\in\Lambda^+$ are trivial. For the remaining inner fine gradings, with universal groups $\ZZ^2\times\ZZ_3^2$, $\ZZ_2^3\times\ZZ_3^2$ and $\ZZ_3^4$, the graded Brauer invariant of $\lambda=\sum_{i=1}^6 m_i\fw_i\in\Lambda^+$ is trivial if $m_1-m_2+m_4-m_5\equiv 0\pmod{3}$ and has support $T$ otherwise, where $T\simeq\ZZ_3^2$ is the 3-torsion subgroup in the case of $\ZZ^2\times\ZZ_3^2$ or $\ZZ_2^3\times\ZZ_3^2$, and the factor $\ZZ_3^2$ associated to the automorphisms $\Theta$ and $\vartheta$ of $\cL$ as in Remark \ref{re:E6gKO} (alternatively, the automorphisms $F_1$ and $F_2$ in the model given in \cite{DV_E6}) in the case of $\ZZ_3^4$.
\end{theorem}

\begin{proof}
It suffices to consider the cases with nontrivial $3$-torsion. Recall from Equations \eqref{eq:bichar_for_Q} and \eqref{eq:def_beta_quasitorus} that $\hat{\beta}^Q_\lambda(x,y)=\Psi_\lambda(\hat{\beta}^Q(x,y))$ for all $x,y\in Q=Q_\lambda$ and $\lambda\in\Lambda^+$, and we have already computed the radical of $\hat{\beta}^Q$. Now, $\Psi_\lambda$ is a character of the center $Z$ of the simply connected group of type $E_6$, i.e., a homomorphism of algebraic groups $Z\to\FF^\times$. The mapping $\Lambda^+\to\wh{Z},\,\lambda\mapsto\Psi_\lambda$, is multiplicative and trivial for $\lambda\in\Lambda^\mathrm{r}$ 
(in fact, it induces an isomorphism $\Lambda/\Lambda^+\stackrel{\sim}{\longrightarrow}\wh{Z}$). 
For type $E_6$, the coset of $\fw_1$ generates $\Lambda/\Lambda^\mathrm{r}\simeq\ZZ_3$, and we have $\fw_1\equiv-\fw_2\equiv\fw_4\equiv-\fw_5$ and $\fw_3\equiv\fw_6\equiv 0$ modulo $\Lambda^\mathrm{r}$. Therefore, $\Psi_\lambda=\Psi_{\fw_1}^{m_1-m_2+m_4-m_5}$ and hence $\hat{\beta}^Q_\lambda=\hat{\beta}_1^{m_1-m_2+m_4-m_5}$ where $\hat{\beta}_1$ is the graded Brauer invariant of $\fw_1$, which is the highest weight of the nontrivial simple $\cL$-module of lowest dimension (namely, the Albert algebra of dimension $27$). Since $\Psi_{\fw_1}$ is injective, the radical of $\hat{\beta}_1$ is the same as the radical of $\hat{\beta}^Q$, namely, $Q_1$. By definition, $T=Q_1^\perp$. The result follows.
\end{proof}

\begin{corollary}
Let $\cL$ be the simple Lie algebra of type $E_6$, equipped with a $G$-grading $\Gamma$ that is induced from one of the inner fine gradings by a homomorphism $\nu\colon G^u\to G$ where $G^u$ is the universal group of the fine grading. If $G^u$ has nontrivial $3$-torsion, let $T\simeq\ZZ_3^2$ be the distinguished subgroup of $G^u$ indicated in Theorem \ref{th:E6_inner}.
\begin{enumerate}
\item[(a)] If $G^u$ has trivial $3$-torsion or $\nu$ is not injective on $T$ then every finite-dimensional $\cL$-module admits a compatible $G$-grading. 
\item[(b)] If $\nu$ is injective on $T$ then a finite-dimensional $\cL$-module $V$ admits a compatible $G$-grading if and only if, for any $\lambda=\sum_{i=1}^6 m_i\fw_i\in\Lambda^+$ satisfying $m_1-m_2+m_4-m_5\not\equiv 0\pmod{3}$, the multiplicity of the simple $\cL$-module $V(\lambda)$ in $V$ is divisible by $3$.
\end{enumerate}
\end{corollary}

\begin{proof}
Recall that the graded Brauer invariant of $\lambda$ with respect to $\Gamma$ is determined by $\hat{\beta}_\lambda(\chi_1,\chi_2)=\hat{\beta}^Q_\lambda(\alpha_{\chi_1},\alpha_{\chi_2})$, for $\chi_1,\chi_2\in\wh{G}$, so Theorem \ref{th:E6_inner} implies that $\hat{\beta}_\lambda$ is trivial if $G^u$ has trivial $3$-torsion, $m_1-m_2+m_4-m_5\equiv 0\pmod{3}$ or the image of the homomorphism $\wh{G}\to Q,\,\chi\mapsto\alpha_\chi$ (which is dual to $G^u\to G$) does not generate $Q$ modulo $T^\perp$ (which is equivalent to saying that $\nu$ is not injective on $T$). In the remaining cases, the support of the Brauer invariant is $\nu(T)\simeq\ZZ_3^2$, so the graded Schur index is $3$.
\end{proof}


\section{Graded Brauer invariants for $E_7$}\label{se:E7}

Let $\cL$ be the simple Lie algebra of type $E_7$. The $14$ different fine gradings of $\cL$, up to equivalence, are listed in \cite[Fig. 6.2]{EK_mon} and are distinguished by their universal groups. However, to compute the graded Brauer invariants for most of these gradings, we find it more convenient to use models different from those mentioned in \cite{EK_mon}. So, we start by describing these models.

\subsection{A construction of $\cL$ from $\frsl_8(\FF)$}\label{ss:E7A7}

Let $V$ be an eight-dimensional vector space. The Lie algebra $\cL$ can be constructed (see \cite[Chapter 12]{Adams}) as the direct sum
\begin{equation}\label{eq:e7_sl}
\cL=\frsl(V)\oplus{\textstyle \bigwedge^4V}.
\end{equation}

Here $\frsl(V)$ is a subalgebra, and the bracket $[f,X]$ for $f\in \frsl(V)$ and $X\in\bigwedge^4V$ is given by the natural action of $\frsl(V)$ on $\bigwedge^4V$: $[f,X]=-[X,f]=f.X$. To define the bracket of two elements in $\bigwedge^4V$, which will be an element of $\frsl(V)$, we fix a determinant map (i.e., a nonzero skew-symmetric multilinear map) $\det:V^8\rightarrow \FF$, or, equivalently, fix a linear isomorphism $\bigwedge^8V\simeq  \FF$, and consider the nondegenerate symmetric bilinear form
\[
\begin{split}
\Omega:{\textstyle \bigwedge^4V}\times {\textstyle \bigwedge^4V}&\longrightarrow \FF,\\
(v_1\wedge v_2\wedge v_3\wedge v_4,w_1\wedge w_2\wedge w_3\wedge w_4)&\mapsto 
\det(v_1,v_2,v_3,v_4,w_1,w_2,w_3,w_4).
\end{split}
\]
$\Omega$ is invariant under the natural action of $\SL(V)$ (equivalently, of $\frsl(V)$) and is uniquely determined by this condition, up to a scalar. 
Then the bracket $[X,Y]$ of two elements $X,Y\in\bigwedge^4V$ is determined by the following formula:
\[
\tr\bigl(f[X,Y])=\Omega(f.X,Y)
\]
for all $f\in \frsl(V)$, where $\tr$ denotes the trace in $\frgl(V)=\End_\FF(V)$.

By construction, the direct sum decomposition \eqref{eq:e7_sl} is a grading by $\ZZ_2$, so the linear map $\sigma$ given by
\begin{equation}\label{eq:sigma}
\sigma(f)=f\text{ for }f\in \frsl(V),\qquad \sigma(X)=-X\text{ for }X\in\bigwedge^4V,
\end{equation}  
is an order $2$ automorphism of $\cL$.

There is a natural homomorphism 
\[
\begin{split}
\Phi:\SL(V)&\longrightarrow \Aut(\cL),\\
   f&\mapsto \begin{cases} d\mapsto fdf^{-1}&\text{for $d\in\frsl(V)$,}\\
                           X\mapsto f\cdot X=\bigl(\textstyle{\bigwedge^4}f\bigr)(X)&\text{for $X\in \bigwedge^4V$.}\end{cases}
\end{split}
\]
Its kernel consists of scalar operators (as they are the elements that commute with the elements of $\frsl(V)$) and hence $\ker\Phi=\langle \bi I\rangle$, where $I$ is the identity element of $\SL(V)$ and $\bi\in\FF$ is a primitive $4$-th root of unity. Since the algebraic group $\SL(V)$ is simply connected, there is a (unique) homomorphism $\tilde\Phi$ such that the following diagram is commutative:
\begin{equation}\label{eq:SLV_E7sc}
\begin{aligned}
\xymatrix{
& \tilde{\cG}\ar@{->>}[d]^{\pi}\\
\SL(V)\ar[r]^{\Phi} \ar[ur]^{\tilde\Phi} & \Aut(\cL)
}
\end{aligned}
\end{equation}
where $\tilde{\cG}$ is the simply connected group of type $E_7$.
Since $\SL(V)$ has rank $7$ (type $A_7$), its image in $\tilde{\cG}$ contains $Z(\tilde{\cG})$, so the kernel of $\tilde\Phi$ is a subgroup of index $2$ in the kernel of $\Phi$ and hence $\ker\tilde\Phi=\langle -I\rangle$.

\subsection{A construction of $\cL$ from $\frsp_8(\FF)$}\label{ss:E7C4}

Now let $b:V\times V\rightarrow \FF$ be a nondegenerate skew-symmetric bilinear form, and consider the associated $\SP(V,b)$-invariant (equivalently, $\frsp(V,b)$-invariant) nondegenerate symmetric bilinear form
\[
\begin{split}
\Omega_b:{\textstyle \bigwedge^4V}\times {\textstyle \bigwedge^4V}&\longrightarrow \FF,\\
(v_1\wedge v_2\wedge v_3\wedge v_4,w_1\wedge w_2\wedge w_3\wedge w_4)&\mapsto 
\det\Bigl(b(v_i,w_j)\Bigr).
\end{split}
\]
Take a symplectic basis $\{x_1,x_2,x_3,x_4,y_1,y_2,y_3,y_4\}$ of $V$, i.e., $b(x_i,y_j)=\delta_{ij}$ (Kronecker's delta) and $b(x_i,x_j)=b(y_i,y_j)=0$ for all $1\leq i,j\leq 4$. Also adjust $\Omega$ so that $\Omega(x_1\wedge x_2\wedge x_3\wedge x_4,y_1\wedge y_2\wedge y_3\wedge y_4)=1$. The space $\bigwedge^4 V$ has a basis consisting of the elements of the form $x_{i_1}\wedge \cdots\wedge x_{i_r}\wedge y_{j_1}\wedge\cdots\wedge y_{j_s}$, where $r,s\geq 0$, $r+s=4$, and $1\leq i_1< \cdots< i_r\leq 4$, $1\leq j_1< \cdots\leq j_s< 4$. We will refer to these elements as `basic wedges'.

Use the above symplectic basis of $V$ to identify $\frsl(V)$ with $\frsl_8(\FF)$ and $\frsp(V,b)$ with $\frsp_8(\FF)$. In particular, let $\frh$ be the Cartan subalgebra of $\frsp(V,b)$ consisting of the diagonal operators relative to this basis, and let $\veps_1,\ldots,\veps_4:\frh\rightarrow \FF$ be the linear maps such that $\veps_i(h)$ is the $i$-th diagonal element of $h$. We order the weights of $\frh$ lexicographically, so $\veps_1>\veps_2>\veps_3>\veps_4$, and this determines a system of simple roots for $\frsp_8(\FF)$, of type $C_4$. The fundamental weights are $\fw_i=\veps_1+\cdots+\veps_i$, $1\leq i\leq 4$. The basic wedge $x_1\wedge x_2\wedge x_3\wedge x_4$ is a maximal vector (i.e., it is killed by all positive root spaces), of weight $\fw_4$, so it generates a $\frsp(V,b)$-submodule of $\bigwedge^4V$ isomorphic to $V(\fw_4)$. In fact, all basic wedges are weight vectors. For instance, $x_2\wedge x_3\wedge x_4\wedge y_3$ has weight $\veps_2+\veps_4$.

Since both $\Omega$ and $\Omega_b$ are symmetric and nondegenerate, there is a symmetric operator $\tau:\bigwedge^4V\rightarrow \bigwedge^4V$, relative to $\Omega$, defined by
\[
\Omega_b(X,Y)=\Omega\bigl(\tau(X),Y\bigr),
\]
for all $X,Y\in\bigwedge^4V$. Note that 
\begin{equation}\label{eq:tau_x1x2x3x4}
\tau(x_1\wedge x_2\wedge x_3\wedge x_4)=x_1\wedge x_2\wedge x_3\wedge x_4,
\end{equation}
because 
\[
\Omega(x_1\wedge x_2\wedge x_3\wedge x_4,y_1\wedge y_2\wedge y_3\wedge y_4)=1=\Omega_b(x_1\wedge x_2\wedge x_3\wedge x_4,y_1\wedge y_2\wedge y_3\wedge y_4)
\] 
and $x_1\wedge x_2\wedge x_3\wedge x_4$ is orthogonal, relative to both $\Omega$ and $\Omega_b$, to any other basic wedge.

Let $f\in\End_\FF(V)$ and let $f^*$ denote the adjoint operator relative to $b$: 
\[
b\bigl(f(v),w\bigr)=b\bigl(v,f^*(w)\bigr),
\] 
for all $v,w\in V$. Then, for any $v_i,w_i\in V$, $1\leq i\leq 4$, and $f\in\GL(V)$, we have:
\[
\begin{split}
&\Omega_b\bigl(f(v_1)\wedge f(v_2)\wedge f(v_3)\wedge f(v_4),w_1\wedge w_2\wedge w_3\wedge w_4\bigr)\\
&=\det\Bigl(b\bigl(f(v_i),w_j\bigr)\Bigr)=\det\Bigl(b\bigl(v_i,f^*(w_j)\bigr)\Bigr),
\end{split}
\]
so $\Omega_b(f\cdot X,Y)=\Omega_b(X,f^*\cdot Y)$ for any $f\in\GL(V)$ and $X,Y\in\bigwedge^4 V$, where $f\cdot X$ denotes the natural action of $\GL(V)$ on $\bigwedge^4 V$. Hence, $\Omega_b(f.X,Y)=\Omega_b(X,f^*.Y)$ for any $f\in\frgl(V)$ and $X,Y\in\bigwedge^4 V$. (Note that the natural actions on $\bigwedge^4V$ of the elements of $\GL(V)$ and of $\frgl(V)$ are different.) Therefore, for any $f\in\SL(V)$ and $X,Y\in\bigwedge^4V$, the invariance of $\Omega$ under the action of $\SL(V)$ gives:
\[
\begin{split}
\Omega\bigl(\tau(f\cdot X),Y\bigr)&=\Omega_b(f\cdot X,Y)=\Omega_b(X,f^*\cdot Y)\\
&=\Omega\bigl(\tau(X),f^*\cdot Y\bigr)=\Omega\bigl((f^*)^{-1}\cdot \tau(X),Y\bigr).
\end{split}
\]
Similarly, $\Omega\bigl(\tau(f.X),Y\bigr)=-\Omega\bigl(f^*.\,\tau(X),Y\bigr)$ for $f\in\frsl(V)$. To summarize:
\begin{align}
\tau(f\cdot X)&=(f^*)^{-1}\cdot\tau(X) \quad\text{for $f\in \SL(V)$},\nonumber \\
\tau(f\,.\,X)&=-(f^*)\,.\,\tau(X) \quad\text{for $f\in\frsl(V)$},\label{eq:tau_ff*}
\end{align}
and all $X\in\bigwedge^4V$. In particular, $\tau$ commutes with the actions of $\SP(V,b)$ and $\frsp(V,b)$.

\begin{proposition}\label{pr:tau}
Extend the above operator $\tau$ on $\bigwedge^4V$ to an operator on $\cL=\frsl(V)\oplus\bigwedge^4V$, denoted by the same letter, by sending $f\mapsto -f^*$ for all $f\in \frsl(V)$. Then $\tau$ is an order $2$ automorphism of the Lie algebra $\cL$ that commutes with $\sigma$.
\end{proposition}

\begin{proof}
Consider the bilinear map $\Omega'$ on $\bigwedge^4V$ given by 
\[
\Omega'(X,Y)=\Omega\bigl(\tau(X),\tau(Y)\bigr).
\] 
For any $f\in\SL(V)$ and $X,Y\in\bigwedge^4V$, we have
\[
\begin{split}
\Omega'(f\cdot X,f\cdot Y)&=\Omega\bigl(\tau(f\cdot X),\tau(f\cdot Y)\bigr)\\
&=\Omega\bigl((f^*)^{-1}\cdot \tau(X),(f^*)^{-1}\cdot \tau(Y)\bigr)=\Omega\bigl(\tau(X),\tau(Y)\bigr)=\Omega'(X,Y),
\end{split}
\]
so $\Omega'$ is also $\SL(V)$-invariant and hence there is a nonzero scalar $\mu\in\FF$ such that $\Omega'=\mu\,\Omega$. This means that $\tau$ is a similarity of the form $\Omega$ (with multiplier $\mu$). But $\tau$ is a symmetric operator, so for any $X,Y\in\bigwedge^4V$, we have:
\[
\mu\,\Omega(X,Y)=\Omega'(X,Y)=\Omega\bigl(\tau(X),\tau(Y)\bigr)=\Omega\bigl(\tau^2(X),Y\bigr),
\]
hence $\tau^2=\mu\,\id$. But we know that $x_1\wedge x_2\wedge x_3\wedge x_4$ is an eigenvector of $\tau$ with eigenvalue $1$, so $\mu=1$.

Now, for any $f\in \frsl(V)$ and $X,Y\in\bigwedge^4V$, we have:
\[
\begin{split}
\tr\bigl(f[\tau(X),\tau(Y)]\bigr)&=\Omega\bigl(f.\tau(X),\tau(Y)\bigr)=-\Omega\bigl(\tau(f^*.X),\tau(Y)\bigr)=-\Omega\bigl(f^*.X,Y\bigr)\\
& = -\tr\bigl(f^*[X,Y]\bigr)=-\tr\bigl([X,Y]^*f\bigr)=-\tr\bigl(f[X,Y]^*\bigr),
\end{split}
\]
hence we obtain $[\tau(X),\tau(Y)]=-[X,Y]^*$, which, together with Equation \eqref{eq:tau_ff*}, finishes the proof.
\end{proof}

There is a natural homomorphism $c:\bigwedge^4V\rightarrow \bigwedge^2V$ of  $\frsp(V,b)$-modules given by `partial contraction' as follows: 
\[
c(v_1\wedge v_2\wedge v_3\wedge v_4)=\sum_{\pi}(-1)^\pi b(v_{\pi(1)},v_{\pi(2)})v_{\pi(3)}\wedge v_{\pi(4)},
\]
where $\pi$ runs over the permutations of $\{1,2,3,4\}$ with $\pi(1)<\pi(2)$ and $\pi(3)<\pi(4)$ and $(-1)^\pi$ denotes the sign of $\pi$. (See, for instance, \cite[\S 17.2]{FH}.) The kernel of $c$ is the simple module $V(\fw_4)$ generated by the basic wedge $x_1\wedge x_2\wedge x_3\wedge x_4$. Also, there is the contraction $\bigwedge^2V\rightarrow \FF$, $v_1\wedge v_2\mapsto b(v_1,v_2)$, whose kernel is $V(\fw_2)$. Thus, we have the following decompositions of $\frsp(V,b)$-modules:
\begin{align*}
\textstyle \bigwedge^4V&\simeq V(\fw_4)\oplus V(\fw_2)\oplus V(0),\\
\frsl(V)&\simeq V(2\fw_1)\oplus V(\fw_2),
\end{align*} 
where $V(2\fw_1)$ is the adjoint module and the copy of $V(\fw_2)$ is given by the subspace of symmetric operators $\{f\in\frsl(V): f^*=f\}$. The trivial $\frsp(V,b)$-submodule of $\bigwedge^4 V$ can be explicitly described as follows. Let $\hat b:\bigwedge^4V\rightarrow \FF$ be the `full contraction' :
\[
\hat b(v_1\wedge v_2\wedge v_3\wedge v_4)=b(v_1,v_2)b(v_3,v_4)-b(v_1,v_3)b(v_2,v_4)+b(v_1,v_4)b(v_2,v_3).
\]
(This is the same as $\frac12$ times the composition of contractions $c$ and $b$.)
Then there is a unique element $\tilde b\in\bigwedge^4V$ such that $\Omega(\tilde b,X)=\hat b(X)$ for all $X$, and the $\frsp(V,b)$-invariance of $\hat b$ is equivalent to the fact that $\tilde b$ is annihilated by $\frsp(V,b)$, so $\FF\tilde b$ is the trivial $\frsp(V,b)$-submodule in $\bigwedge^4V$. One checks at once that $\hat b$ vanishes on all basic wedges except $x_i\wedge x_j\wedge y_i\wedge y_j$ for $1\leq i<j\leq 4$, and we get 
\[
\tilde b=-\sum_{1\leq i<j\leq 4}x_i\wedge x_j\wedge y_i\wedge y_j.
\]

\begin{lemma}\label{le:tau}
The order $2$ automorphism $\tau$ of $\cL$ is the identity on $\frsp(V,b)$ and on the copies of $V(\fw_4)$ and $V(0)$ in $\bigwedge^4V$, and it is minus the identity on the space $\{f\in\frsl(V):f^*=f\}$ and on the copy of $V(\fw_2)$ in $\bigwedge^4V$.
\end{lemma}

\begin{proof}
The action of $\tau$ on $\frsl(V)$ is clear. In particular, $\tau$ is a homomorphism of $\frsp(V,b)$-modules. Equation \eqref{eq:tau_x1x2x3x4} implies that the restriction of $\tau$ to $V(\fw_4)$ is the identity. For any  $i<j$, a straightforward computation gives:
\[
\tau(x_i\wedge x_j\wedge y_i\wedge y_j)=x_{i'}\wedge x_{j'}\wedge y_{i'}\wedge y_{j'},
\]
where $i'<j'$ and $\{i,j,i',j'\}=\{1,2,3,4\}$. It follows that $\tau(\tilde b)=\tilde b$ and $\tau$ is not the identity on $\bigwedge^4V$. 
By Schur's Lemma, the restriction of $\tau$ to the copy of $V(\fw_2)$ in $\bigwedge^4V$ is a scalar, so it has to be $-1$.
\end{proof}

As the order $2$ automorphisms $\sigma$ in Equation \eqref{eq:sigma} and $\tau$ in Proposition \ref{pr:tau} commute with one another, they induce a $\ZZ_2^2$-grading on $\cL$ with the following components:
\begin{equation}\label{eq:Z22onE7}
\begin{split}
\cL_{(\bar 0,\bar 0)}&=\frsp(V,b)\simeq V(2\fw_1),\\
\cL_{(\bar 0,\bar 1)}&=\{ f\in\frsl(V): f^*=f\}\simeq V(\fw_2),\\
\cL_{(\bar 1,\bar 0)}&=\{X\in {\textstyle \bigwedge^4V}: \tau(X)=X\}\simeq V(\fw_4)\oplus V(0),\\
\cL_{(\bar 1,\bar 1)}&=\{X\in {\textstyle \bigwedge^4V}: \tau(X)=-X\}\simeq V(\fw_2).
\end{split}
\end{equation}
Thus, $\cL$ is expressed in terms of $\frsl_8(\FF)$ and its modules. Restricting the action of $\SL(V)$ on $\cL$, we obtain a homomorphism $\SP(V,b)\to\Aut(\cL)$. The kernel is $\langle \bi I\rangle\cap\SP(V,b)=\langle-I\rangle$, so we obtain an embedding $\PSP(V,b)\hookrightarrow\Aut(\cL)$. We conclude this discussion by the following result, which will allow us to deal with $7$ fine gradings of $\cL$ in one stroke.

\begin{theorem}\label{th:e7c4}
Let $\cL$ be the simple Lie algebra of type $E_7$ as constructed in \eqref{eq:e7_sl} and let $\sigma$ and $\tau$ be the two commuting order $2$ automorphisms defined by Equation~\eqref{eq:sigma} and Proposition~\ref{pr:tau}, respectively, which give us the $\ZZ_2^2$-grading \eqref{eq:Z22onE7}.
\begin{romanenumerate}
\item
The subgroup $\langle\sigma,\tau\rangle$ of $\Aut(\cL)$ is nontoral. Hence, $\hat{\beta}^P(\sigma,\tau)\ne 1$ for any quasitorus $P$ in $\Aut(\cL)$ containing $\sigma$ and $\tau$.
\item 
The centralizer in $\Aut(\cL)$ of the subgroup $\langle\sigma,\tau\rangle$ is the direct product of itself and the copy of $\PSP(V,b)$ embedded into $\Aut(\cL)$ as described above:
\[
\Centr_{\Aut(\cL)}\bigl(\langle \sigma,\tau\rangle\bigr)=\langle \sigma,\tau\rangle\times \PSP(V,b).
\]
\item
For any maximal quasitorus $Q_0$ in $\PSP(V,b)$, the quasitorus $Q\bydef\langle \sigma,\tau\rangle\times Q_0$ is maximal in $\Aut(\cL)$, and $Q_0$ is the radical of $\hat{\beta}^Q$. 
\end{romanenumerate}
\end{theorem}

\begin{proof}
The subalgebra of elements fixed by $\sigma$ and $\tau$ is $\frsp(V,b)$ of rank $4$, so (i) follows by Lemma \ref{lm:3}.

Since $\PSP(V,b)$ preserves the homogeneous components of the $\ZZ_2^2$-grading \eqref{eq:Z22onE7}, it is contained in $\Centr_{\Aut(\cL)}\bigl(\langle \sigma,\tau\rangle\bigr)$. Since $\sigma$ and $\tau$ act trivially on $\Sp(V,b)$, the intersection of $\langle \sigma,\tau\rangle$ and $\PSP(V,b)$ is trivial. The map
\[
\begin{split}
\Centr_{\Aut(\cL)}\bigl(\langle \sigma,\tau\rangle\bigr)&\longrightarrow \Aut\bigl(\frsp(V,b)\bigr)\simeq \PSP(V,b),\\
f\quad&\mapsto\quad f\vert_{\frsp(V,b)},
\end{split}
\]
is a split surjection with kernel $\{\varphi\in\Centr_{\Aut(\cL)}\bigl(\langle \sigma,\tau\rangle\bigr):
 \varphi\vert_{\frsp(V,b)}=\id\}$. If $\varphi$ is in this kernel, by Schur's Lemma, there are nonzero scalars $\lambda,\mu\in\FF$ such that $\varphi\vert_{\cL_{(\bar 0,\bar 1)}}=\lambda\,\id$, $\varphi\vert_{\cL_{(\bar 1,\bar 1)}}=\mu\,\id$. But $0\ne [\cL_{(\bar 0,\bar 1)},\cL_{(\bar 0,\bar 1)}]\subseteq \frsp(V,b)$, so  $\lambda^2=1$. In the same vein, $\mu^2=1$. Since $\cL_{(\bar 0,\bar 1)}$ and $\cL_{(\bar 1,\bar 1)}$ generate $\cL$, it follows that $\varphi$ is determined by $\lambda$ and $\mu$, and hence it belongs to $\langle\sigma,\tau\rangle$, completing the proof of (ii).

The first part of (iii) is clear. The homomorphism $\tilde{\Phi}$ of \eqref{eq:SLV_E7sc} has kernel $\langle-I\rangle$ and hence induces an embedding $\iota:\PSP(V,b)\hookrightarrow\tilde{\cG}$, which completes the following commutative diagram:
\[
\xymatrix{
& \tilde{\cG}\ar@{->>}[d]^{\pi}\\
\PSP(V,b)\ar@{^{(}->}[r] \ar[ur]^{\iota} & \Aut(\cL)
}
\]
where, as before, $\tilde{\cG}$ is the simply connected group of type $E_7$. Since $\pi^{-1}\bigl(\PSP(V,b)\bigr)=Z(\tilde{\cG})\times\iota\bigl(\PSP(V,b)\bigr)$, we obtain
\[
\pi^{-1}\Bigl(\Centr_{\Aut(\cL)}\bigl(\langle\sigma,\tau\rangle\bigr)\Bigr)
=\pi^{-1}\bigl(\langle \sigma,\tau\rangle\bigr)\rtimes \iota\bigl(\PSP(V,b)\bigr).
\]
The action of the connected group $\iota\bigl(\PSP(V,b)\bigr)$ on the finite group $\pi^{-1}\bigl(\langle\sigma,\tau\rangle\bigr)$ is necessarily trivial, so we actually have a direct product. It follows that, for any maximal quasitorus $Q_0$ in $\PSP(V,b)$, we have
\[
\pi^{-1}\bigl(\langle \sigma,\tau\rangle\times Q_0\bigr)=\pi^{-1}\bigl(\langle\sigma,\tau\rangle\bigr)\times \iota(Q_0),
\]
which proves that $Q_0$ lies in the radical of $\hat\beta^Q$. Now part (i) completes the proof of (iii).
\end{proof}

\subsection{Fine gradings}\label{ss:E7_main}

Now we are going to compute the graded Brauer invariants for the $14$ fine gradings of $\cL$, the simple Lie algebra of type $E_7$. We divide them into six cases, labeled by their universal groups, according to the method of computation (in particular, what model of $\cL$ will be used).

\medskip

$\boxed{\ZZ^4\times\ZZ_2^2,\, \ZZ^2\times\ZZ_2^4,\, \ZZ\times\ZZ_2^5,\, \ZZ\times\ZZ_2^6,\, \ZZ_2^7,\, \ZZ_2^5\times\ZZ_4,\text{ and }\ZZ_2^8}$\quad
Up to conjugacy, there are seven maximal quasitori in $\PSP_8(\FF)\simeq\Aut(\Sp_8(\FF))$ (see \cite[Example 5.3]{EldFine} or \cite[Example 3.72]{EK_mon}), so Theorem \ref{th:e7c4}(iii) gives us seven fine gradings on $\cL$, whose universal groups are
\[
\ZZ_2^2\times \begin{cases}
\ZZ^4\\ \ZZ_2^2\times\ZZ^2 \\ \ZZ_2^3\times \ZZ\\ 
\ZZ_2^4\times\ZZ\\ \ZZ_2^5\\  \ZZ_2^3\times\ZZ_4\\ \ZZ_2^6
\end{cases}
\]
as stated above. In all cases, there is a distinguished subgroup $T\simeq\ZZ_2^2$ associated to the automorphisms $\sigma$ and $\tau$ (the first factor); $T^\perp$ is the radical of $\hat{\beta}^Q$. 

\medskip

$\boxed{\ZZ^7\text{ and }\ZZ\times\ZZ_3^3}$\quad
There is no $2$-torsion here, so the Brauer invariants are trivial by Corollary \ref{cor:Brauer_E6E7}.

\medskip

$\boxed{\ZZ_2^3\times\ZZ_4^2\text{ and }\ZZ\times\ZZ_2\times\ZZ_4^2}$\quad
Here we go back to the model of $\cL$ in terms of $\Sl_8(\FF)$, as described in \eqref{eq:e7_sl}. 

\begin{lemma}\label{le:z42s}
Recall the commuting order $2$ automorphisms $\sigma$ and $\tau$ of $\cL$ defined by Equation \eqref{eq:sigma} and Proposition \ref{pr:tau}, respectively, and the homomorphisms $\Phi:\SL(V)\to\Aut(\cL)$ and $\tilde{\Phi}:\SL(V)\to\tilde{\cG}$ in \eqref{eq:SLV_E7sc}.
\begin{romanenumerate}
\item The centralizer of $\sigma$ in $\Aut(\cL)$ is the following:
\[
\Centr_{\Aut(\cL)}(\sigma)=\Phi\bigl(\SL(V)\bigr)\rtimes \langle\tau\rangle \simeq \SL(V)/\langle \bi I\rangle \rtimes \ZZ_2.
\]
\item Let $\tilde\sigma\in \tilde{\cG}$ be a preimage of $\sigma$ under $\pi:\tilde{\cG}\to\Aut(\cL)$. Then 
\[
\Centr_{\tilde{\cG}}(\tilde{\sigma})=\tilde{\Phi}\bigl(\SL(V)\bigr) \simeq \SL(V)/\langle -I\rangle.
\]
\end{romanenumerate}
\end{lemma}

\begin{proof}
It is clear that $\tau$ and the image of $\Phi$ are contained in $\Centr_{\Aut(\cL)}(\sigma)$. Let $\varphi$ be an automorphism of $\cL$ commuting with $\sigma$. Then $\varphi$ induces an automorphism of $\frsl(V)$. Composing with $\tau$ if needed, we may assume that $\varphi$ induces an inner automorphism of $\frsl(V)$. Hence there exists an element $f\in\SL(V)$ such that $\varphi(d)=fdf^{-1}$ for all $d\in\frsl(V)$. Then the restriction of $\varphi\circ\Phi(f)^{-1}$ to $\frsl(V)$ is the identity. By Schur's Lemma, this implies that $\varphi\circ\Phi(f)^{-1}$ is either identity or $\sigma$. But $\sigma=\Phi(\veps I)$ where $\veps$ is a primitive $8$-th root of unity. This completes the proof of (i).

The centralizer $\Centr_{\tilde{\cG}}(\tilde\sigma)$ is contained in $\pi^{-1}\bigl(\Centr_{\Aut(\cL)}(\sigma)\bigr)=\tilde{\Phi}\bigl(\SL(V)\bigr)\langle\tilde\tau\rangle$ by part (i), where $\tilde\tau$ is a preimage of $\tau$ under $\pi$. We may take $\tilde\sigma=\tilde\Phi(\veps I)$ and, by Theorem \ref{th:e7c4}(i), we already know that $\tilde\sigma$ and $\tilde\tau$ do not commute. Part (ii) follows.
\end{proof}

The restriction map
\[
\begin{split}
\varrho: \Centr_{\Aut(\cL)}(\sigma)&\longrightarrow \Aut(\frsl(V)),\\
       \varphi\quad &\mapsto\quad \varphi\vert_{\frsl(V)},
\end{split}
\]
is a surjective homomorphism with kernel $\langle\sigma\rangle$. Let $P$ be a maximal quasitorus of $\Aut(\frsl(V))$ such that $\varrho^{-1}(P)$ is abelian. Then $\varrho^{-1}(P)$ is a maximal quasitorus of $\Centr_{\Aut(\cL)}(\sigma)$, and hence of $\Aut(\cL)$, as any quasitorus containing $\sigma$ is contained in its centralizer. Through $\Phi$, we identify $\SL(V)/\langle \bi I\rangle$ with a subgroup of $\Centr_{\Aut(\cL)}(\sigma)$.

To construct a maximal quasitorus in $\Aut(\frsl(V))$, we factor $V$ as a tensor product $V=V_1\otimes V_2$ of a four-dimensional vector space $V_1$ and a two-dimensional vector space $V_2$. Let $X_1,Y_1$ be the generalized Pauli matrices (see Equation \eqref{Pauli}) for $V_1$, so $X_1^4=Y_1^4=I_1$ and $X_1Y_1=\bi Y_1X_1$, and let $X_2,Y_2$ be those for $V_2$, so $X_2^2=Y_2^2=I_2$ and $X_2Y_2=-Y_2X_2$, where $I_1$ and $I_2$ are the identity operators on $V_1$ and $V_2$, respectively. Note that the operators
\begin{equation}\label{eq:Z42Z23}
X_1\otimes I_2,\, Y_1\otimes I_2,\, I_1\otimes X_2,\, I_1\otimes Y_2,\text{ and }\veps I_1\otimes I_2
\end{equation}
are in $\SL(V)$. Their classes modulo $\langle \bi I\rangle$ generate a quasitorus $Q\simeq \ZZ_4^2\times\ZZ_2^3$ in $\SL(V)/\langle \bi I\rangle$, whose image under $\varrho$ is a maximal quasitorus $P\simeq\ZZ_4^2\times\ZZ_2^2$ in $\Aut(\frsl(V))$ --- see \cite[\S 3.3]{EK_mon} and also recall that $\veps I_1\otimes I_2=\veps I$ corresponds to $\sigma$. Hence, $Q=\varrho^{-1}(P)$ is a maximal quasitorus of $\Aut(\cL)$, and it gives rise to a fine grading on $\cL$ with universal group $\ZZ_4^2\times\ZZ_2^3$. 

Now use $\tilde\Phi$ to identify $\SL(V)/\langle -I\rangle$ with a subgroup of $\tilde{\cG}$. Then the classes of the elements in \eqref{eq:Z42Z23} modulo $\langle -I\rangle$ are preimages under $\pi$ of the corresponding generators of $Q$. It follows at once that the radical of $\hat\beta^Q$ is given by
\begin{equation}\label{eq:kerbZ42Z23}
\rad\hat\beta^Q=\left\langle [X_1^2\otimes I_2],[Y_1^2\otimes I_2]\right\rangle\times\left\langle [I_1\otimes X_2],[I_1\otimes Y_2]\rangle\times\langle [\veps I_1\otimes I_2]\right\rangle,
\end{equation}
where brackets denote classes modulo $\langle \bi I\rangle$. 
(In particular, we use the fact that $\hat\beta^Q\bigl([X_1\otimes I_2],[Y_1\otimes I_2]\bigr)=\tilde\Phi(\bi I)\ne 1$.)

In the same vein, we construct another maximal quasitorus of $\Aut(\cL)$ by taking the operators
\begin{equation*} 
X_1\otimes I_2,\, Y_1\otimes I_2,\, 
I_1\otimes \left(\begin{smallmatrix} t&0\\ 0&t^{-1}\end{smallmatrix}\right)\text{ for all } t\in\FF^\times,\text{ and }\veps I_1\otimes I_2,
\end{equation*}
in $\SL(V)$. Their classes modulo $\langle \bi I\rangle$ generate a quasitorus $Q'\simeq\ZZ_4^2\times\FF^\times\times\ZZ_2$ in $\SL(V)/\langle \bi I\rangle$, whose image under $\varrho$ is a maximal quasitorus $P'\simeq\ZZ_4^2\times\FF^\times$ in $\Aut(\frsl(V))$. Hence, $Q'=\varrho^{-1}(P')$ is a maximal quasitorus of $\Aut(\cL)$, and it gives rise to a fine grading on $\cL$ with universal group $\ZZ_4^2\times\ZZ\times\ZZ_2$. In this case, we get:
\[
\rad\hat\beta^{Q'}=\left\langle [X_1^2\otimes I_2],[Y_1^2\otimes I_2]\right\rangle\times
\left\{\bigl[I_1\otimes \left(\begin{smallmatrix} t&0\\ 0&t^{-1}\end{smallmatrix}\right)\bigr]: t\in\FF^\times\right\}\times
\left\langle [\veps I_1\otimes I_2]\right\rangle.
\]
For both fine gradings, there is a distinguished subgroup $T\simeq\ZZ_2^2$ of the universal group (contained in the factor $\ZZ_4^2$) such that $T^\perp=\rad\hat{\beta}^Q$, respectively $T^\perp=\rad\hat{\beta}^{Q'}$.

\medskip

$\boxed{\ZZ^3\times\ZZ_2^3}$\quad 
This grading is easily described using Tits construction $\cT(\cC,\cJ)$, where $\cC$ is the Cayley algebra and $\cJ$ is the simple Jordan algebra of Hermitian $3\times 3$ matrices over the quaternion algebra. Recall again that $\cT(\cC,\cJ)$ is defined as the vector space direct sum
\[
\cL=\cT(\cC,\cJ)=\Der(\cC)\oplus(\cC_0\otimes\cJ_0)\oplus\Der(\cJ),
\]
where the subscript $0$ refers to the subspace of trace zero elements (see e.g. \cite[\S 6.2]{EK_mon}). 
In particular, the group $\Aut(\cC)\times\Aut(\cJ)$ is naturally embedded in $\Aut(\cL)$.
The grading in question is obtained by combining the fine $\ZZ^3$-grading on $\cJ$ (the Cartan grading, i.e., coming from the maximal torus of $\Aut(\cJ)\simeq\PSP_6(\FF)$, of type $C_3$) and the fine $\ZZ_2^3$-grading on $\cC$ (obtained from the triple iteration of the Cayley-Dickson doubling process). 
Thus, we are dealing with the quasitorus $Q=Q_1\times Q_2$ where $Q_2$ is a torus in $\Aut(\cJ)$ and $Q_1$ is contained in the simply connected group $\Aut(\cC)$ (type $G_2$). It follows, by Lemma \ref{lm:2}, that $\hat{\beta}^Q$ is trivial. 

\medskip

$\boxed{\ZZ_2^2\times\ZZ_3^3}$\quad Here we think of $\cL$ in terms of Tits construction $\cT(\cH,\cA)$, where $\cH$ is the quaternion algebra and $\cA$ is the Albert algebra, i.e., the Jordan algebra of $3\times 3$ Hermitian matrices over the Cayley algebra $\cC$:
\[
\cL=\cT(\cH,\cA)=\Der(\cH)\oplus(\cH_0\otimes\cA_0)\oplus\Der(\cA).
\]
Similarly to the previous case, the group $\Aut(\cH)\times\Aut(\cA)$ embeds into $\Aut(\cL)$ and the maximal quasitorus associated to our grading has the form $Q=Q_1\times Q_2$ where $Q_1$ lives in $\Aut(\cH)\simeq\PSL_2(\FF)$ and gives the $\ZZ_2^2$-grading on $\cH\simeq M_2(\FF)$ coming from the Pauli matrices (or, alternatively, from the double iteration of the Cayley-Dickson doubling process), and $Q_2$ lives in $\Aut(\cA)$ and gives the $\ZZ_3^3$-grading on $\cA$ that we already encountered in the proof of Theorem \ref{th:E6_outer}. The subalgebra of elements fixed by $Q_1$ is $\Der(\cA)$, of rank $4$ (type $F_4$), so $Q_1$ is not toral (Lemma \ref{lm:3}). We conclude that $Q_2$ is the radical of $\hat\beta^Q$. The distinguished subgroup $T=Q_2^\perp$ of the universal group $\ZZ_2^2\times\ZZ_3^3$ is just its $2$-torsion.

\medskip

$\boxed{\ZZ_2\times\ZZ_4^3}$\quad This last grading is best described in terms of the so-called \emph{Brown algebra}, which is a $56$-dimensional algebra with involution $(\cB,-)$, an example of a structurable algebra. We will follow here \cite{AEK} and \cite{Garibaldi}.

Given a unital algebra with involution $(\cA,-)$, consider the operators $V_{x,y},T_x\in\End_\FF(\cA)$ defined by
\[
V_{x,y}(z)\bydef\{ x,y,z\}=(x\bar y)z+(z\bar y)x-(z\bar x)y,\quad T_x\bydef V_{x,1}.
\]
$(\cA,-)$ is said to be a \emph{structurable algebra} if 
\[
[V_{x,y},V_{z,t}]=V_{V_{x,y}(z),t}-V_{z,V_{y,x}(t)}
\]
for all $x,y,z,t\in\cA$. This reduces to the definition of Jordan algebra if we take the involution to be the identity. The Brown algebra $\cB$ is just one step away from this situation, with the space of skew-symmetric elements being one-dimensional. Fix a skew-symmetric element $s_0\in\cB$ normalized to satisfy $s_0^2=1$. Then we can define a skew-symmetric bilinear form $(.\,,.)$ on $\cB$ by setting $x\bar y-y\bar x = (x,y) s_0$ for all $x,y\in\cB$; this form turns out to be nondegenerate. It is known that the simply connected group $\tilde\cG$ of type $E_7$ can be identified with $\SP\bigl(\cB,(.\,,.)\bigr)\cap\Aut\bigl(\cB,\{.\,,.\,,.\}\bigr)$, i.e., the group of all symplectic transformations that are, at the same time, automorphisms of the triple product. The algebra $\cB$ also has a symmetric bilinear trace form $\langle.\,,.\rangle$. Consider the subspaces $\cS=\{x\in\cB: \bar x=-x\}=\FF s_0$, $\cH=\{x\in\cB: \bar x=x\}$, $\cB_0=\{x\in\cB: \langle x,1\rangle=0\}$, and $\cH_0=\cH\cap\cB_0$. The \emph{structure Lie algebra} of $\cB$ is defined by 
\[
\frstr(\cB)=\espan{V_{x,y}: x,y\in\cB}=\Der(\cB)\oplus T_\cB\subset \frgl(\cB).
\]
Its derived subalgebra $\cL=\frstr(\cB)_0\bydef \Der(\cB)\oplus T_{\cB_0}$ is simple of type $E_7$.

The group of (involution-preserving) automorphisms $\Aut(\cB)$ is the semidirect product of the simply connected group of type $E_6$ (consisting of the automorphisms that fix $s_0$) and the cyclic group of order $2$ (generated by an automorphism that sends $s_0\mapsto-s_0$). Hence, $\Der(\cB)$ is the simple Lie algebra of type $E_6$. Since any automorphism of $\cB$ preserves the bilinear form $(.\,,.)$ and the triple product $\{.\,,.\,,.\}$, the group $\Aut(\cB)$ embeds into $\tilde\cG$. The natural projection $\pi:\tilde\cG\rightarrow\Aut(\cL)$ restricts to the natural action of $\Aut(\cB)$ on $\cL=\frstr(\cB)_0$: 
\[
\begin{split}
\pi\vert_{\Aut(\cB)}:\Aut(\cB)&\hookrightarrow \Aut(\cL)\\
  f\ &\mapsto\begin{cases} d\mapsto fdf^{-1}&\text{for }d\in \Der(\cB),\\
                           T_x\mapsto T_{f(x)}&\text{for }x\in\cB_0.
             \end{cases}
\end{split}
\]
By construction, $\cL=\frstr(\cB)_0$ is $\ZZ_2$-graded, with $\cL\subo=\Der(\cB)\oplus T_{\cS}$ and $\cL\subuno=T_{\cH_0}$. Let $\upsilon$ be the corresponding  automorphism: $\upsilon\vert_{\cL\subo}=\id$ and $\upsilon\vert_{\cL\subuno}=-\id$. 

The Brown algebra has a fine grading with universal group $\ZZ_4^3$. The corresponding quasitorus $\tilde Q_0$ is generated by three order $4$ elements $\sigma_1,\sigma_2,\sigma_3$, where the action of $\sigma_1$ on $\Der(\cB)$ is not inner, but the actions of $\sigma_1^2$, $\sigma_2$ and $\sigma_3$ are inner. Consider $Q_0=\pi(\tilde Q_0)$. This is a quasitorus in $\Aut(\cL)$ that does not contain $\upsilon$ but commutes with it, and the direct product $Q\bydef Q_0\times\langle\upsilon\rangle$ is the maximal quasitorus of $\Aut(\cL)$ giving the desired fine grading on $\cL$ with universal group $\ZZ_4^3\times\ZZ_2$.

\begin{proposition}\label{pr:Brown}
Consider the quasitorus $Q=Q_0\times\langle\upsilon\rangle\cong\ZZ_4^3\times\ZZ_2$ in $\Aut(\cL)$.
\begin{romanenumerate}
\item The subgroup $\langle\sigma_1,\upsilon\rangle$ of $\Aut(\cL)$ is nontoral, so $\hat\beta^Q(\sigma_1,\tau)\ne 1$.
\item $\rad\hat\beta^Q=\langle \sigma_1^2,\sigma_2,\sigma_3\rangle$, 
which is the subgroup of the elements of $Q_0$ that fix $s_0$ (equivalently, whose action on $\Der(\cB)$ is inner).
\end{romanenumerate}
\end{proposition}

\begin{proof}
We first note that $Q_0$ is isotropic for $\hat\beta^Q$, as it is the image under the natural projection $\pi:\tilde\cG\rightarrow \Aut(\cL)$ of the abelian subgroup $\tilde Q_0$. 

Now consider $f\in\Aut(\cB)$. Then $\vphi\bydef\pi(f)$ commutes with $\upsilon$, and we have 
\[
\Fix_{\cL}(\vphi,\upsilon)=\Fix_{\cL_{\subo}}(\vphi)=
\begin{cases}
\Fix_{\Der(\cB)}(\vphi)\oplus T_{\cS}&\text{if }f(s_0)=s_0,\\
\Fix_{\Der(\cB)}(\vphi)&\text{if }f(s_0)=-s_0.
\end{cases}
\]
Suppose $f$ is semisimple. Then $\Fix_{\Der(\cB)}(\vphi)$ is a reductive Lie algebra. If $f(s_0)=s_0$ then $\vphi$ acts on $\Der(\cB)$ as an inner automorphism, so $\Fix_{\Der(\cB)}(\vphi)$ has rank $6$ and hence $\Fix_{\cL}(\vphi,\upsilon)$ has rank $7$ ($T_\cS$ being the one-dimensional center of $\cL\subo$). By Lemma \ref{lm:3}, this means that the subgroup $\langle\vphi,\upsilon\rangle$ of $\Aut(\cL)$ is toral. On the other hand, if $f(s_0)=-s_0$ then $\vphi$ acts on $\Der(\cB)$ as an outer automorphism and hence $\Fix_{\cL}(\vphi,\upsilon)$ has rank $<6$, so the subgroup $\langle\vphi,\upsilon\rangle$ of $\Aut(\cL)$ is nontoral. Both parts (i) and (ii) follow.
\end{proof}

Let $h$ be the degree of the element $s_0$ with respect to our $\ZZ_4^3$-grading on the Brown algebra $\cB$. It is the distinguished element of the corresponding $\ZZ_4^3$-grading on the Lie algebra $\Der(\cB)$ of type $E_6$, i.e., it is the order $2$ element characterized by the property that the grading on $\Der(\cB)$ induced by the natural homomorphism $\ZZ_4^3\to\ZZ_4^3/\langle h\rangle$ is inner. Proposition \ref{pr:Brown} tells us that $\rad\hat\beta^Q=T^\perp$ where $T$ is the subgroup of the universal group $\ZZ_2\times\ZZ_4^3$ generated by the factor $\ZZ_2$ and the distinguished element $h$ in the factor $\ZZ_4^3$.

\medskip

Finally, we are going to summarize our results on fine gradings for $E_7$. We will fix the following enumeration of simple roots $\{\alpha_1,\ldots,\alpha_7\}$ and the corresponding fundamental weights $\{\fw_1,\ldots,\fw_7\}$:

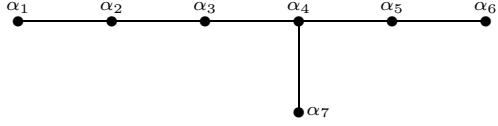
\[\xymatrix{
\bullet\ar@{-}[]+0;[r]+0^<{\alpha_1}  & \bullet\ar@{-}[]+0;[r]+0^<{\alpha_2} & \bullet\ar@{-}[]+0;[r]+0^<{\alpha_3} & \bullet\ar@{-}[]+0;[r]+0^<{\alpha_4}\ar@{-}[]+0;[d]+0^>{\alpha_7}  & \bullet\ar@{-}[]+0;[r]+0^<{\alpha_5} & \bullet\ar@{}[]+0;[r]^<{\alpha_6} & \\
& & & \bullet & &
}
\]

\begin{theorem}\label{th:E7}
Let $\cL$ be the simple Lie algebra of type $E_7$. For the fine gradings on $\cL$ with universal groups $\ZZ^7$, $\ZZ^3\times\ZZ_2^3$ and $\ZZ\times\ZZ_3^3$, the graded Brauer invariants of all $\lambda\in\Lambda^+$ are trivial. For the remaining fine gradings, the graded Brauer invariant of $\lambda=\sum_{i=1}^7 m_i\fw_i\in\Lambda^+$ is trivial if $m_1+m_3+m_7\equiv 0\pmod{2}$ and has support $T$ otherwise, where $T\simeq\ZZ_2^2$ is the distinguished subgroup of the universal group: the 2-torsion subgroup in the case of $\ZZ_2^2\times\ZZ_3^3$, the $2$-periodic elements inside $\ZZ_4^2$ in the case of $\ZZ\times\ZZ_2\times\ZZ_4^2$ or $\ZZ_2^3\times\ZZ_4^2$, the subgroup generated by the factor $\ZZ_2$ and the distinguished element $h\in\ZZ_4^3$ in the case of $\ZZ_2\times\ZZ_4^3$, and the factor $\ZZ_2^2$ associated to the automorphisms $\sigma$ and $\tau$ of $\cL$ as in Theorem \ref{th:e7c4} in the remaining cases of $\ZZ^4\times\ZZ_2^2$,  $\ZZ^2\times\ZZ_2^4$, $\ZZ\times\ZZ_2^5$, $\ZZ\times\ZZ_2^6$, $\ZZ_2^7$, $\ZZ_2^5\times\ZZ_4$ or $\ZZ_2^8$.
\end{theorem}

\begin{proof}
We have already obtained the radical of $\hat{\beta}^Q$ and its perp in all cases, so it remains to invoke the formula  $\hat{\beta}^Q_\lambda(x,y)=\Psi_\lambda(\hat{\beta}^Q(x,y))$, for all $x,y\in Q$ and $\lambda\in\Lambda^+$, and compute the character $\Psi_\lambda$ of the center $Z$ of the simply connected group of type $E_7$. Recall that the mapping $\Lambda^+\to\wh{Z},\,\lambda\mapsto\Psi_\lambda$, is multiplicative and trivial for $\lambda\in\Lambda^\mathrm{r}$. 
For type $E_7$, the coset of $\fw_1$ generates $\Lambda/\Lambda^\mathrm{r}\simeq\ZZ_2$, and we have $\fw_1\equiv\fw_3\equiv\fw_7$ and $\fw_2\equiv\fw_4\equiv\fw_5\equiv\fw_6\equiv 0$ modulo $\Lambda^\mathrm{r}$. 
Therefore, $\Psi_\lambda=\Psi_{\fw_1}^{m_1+m_3+m_7}$ and hence $\hat{\beta}^Q_\lambda=\hat{\beta}_1^{m_1+m_3+m_7}$ where $\hat{\beta}_1$ is the graded Brauer invariant of $\fw_1$, which is the highest weight of the nontrivial simple $\cL$-module of lowest dimension (namely, the Brown algebra of dimension $56$). Since $\Psi_{\fw_1}$ is injective, the radical of $\hat{\beta}_1$ is the same as the radical of $\hat{\beta}^Q$. The result follows.
\end{proof}

\begin{corollary}
Let $\cL$ be the simple Lie algebra of type $E_7$, equipped with a $G$-grading $\Gamma$ that is induced from one of the fine gradings by a homomorphism $\nu\colon G^u\to G$ where $G^u$ is the universal group of the fine grading. If $G^u$ is not isomorphic to $\ZZ^7$, $\ZZ^3\times\ZZ_2^3$ or $\ZZ\times\ZZ_3^3$, let $T\simeq\ZZ_2^2$ be the distinguished subgroup of $G^u$ indicated in Theorem \ref{th:E7}. 
\begin{enumerate}
\item[(a)] If $G^u$ is isomorphic to $\ZZ^7$, $\ZZ^3\times\ZZ_2^3$ or $\ZZ\times\ZZ_3^3$ or if $\nu$ is not injective on $T$, then every finite-dimensional $\cL$-module admits a compatible $G$-grading. 
\item[(b)] If $\nu$ is injective on $T$ then a finite-dimensional $\cL$-module $V$ admits a compatible $G$-grading if and only if, for any $\lambda=\sum_{i=1}^7 m_i\fw_i\in\Lambda^+$ satisfying $m_1+m_3+m_7\not\equiv 0\pmod{2}$, the multiplicity of the simple $\cL$-module $V(\lambda)$ in $V$ is divisible by $2$.\qed
\end{enumerate}
\end{corollary}

\bigskip

%


\end{document}